\documentclass[11pt]{article}
\usepackage{amsmath}
\usepackage{dsfont}
\usepackage{mathrsfs}
\usepackage{amsmath,amssymb}
\usepackage{amsfonts}
\usepackage{hyperref}
\usepackage{amsthm}
\usepackage{graphicx}
\usepackage{subfigure}
\usepackage{xcolor}

\usepackage{bbm}
\usepackage{mathrsfs}
\usepackage{txfonts}
\usepackage{color}
\usepackage{pict2e}

\theoremstyle{plain}
\newtheorem{theorem}{Theorem}[section]
\newtheorem{lemma}[theorem]{Lemma}

\newtheorem{proposition}[theorem]{Proposition}
\newtheorem{corollary}[theorem]{Corollary}

\newtheorem{Bounded Diameter Lemma}[theorem]{Bounded Diameter Lemma}
\theoremstyle{definition}

\newtheorem{remark}[theorem]{Remark}

\newcommand{\Hmm}[1]{\leavevmode{\marginpar{\tiny%
$\hbox to 0mm{\hspace*{-0.5mm}$\leftarrow$\hss}%
\vcenter{\vrule depth 0.1mm height 0.1mm width \the\marginparwidth}%
\hbox to
0mm{\hss$\rightarrow$\hspace*{-0.5mm}}$\\\relax\raggedright #1}}}

\def\R{\mathbb{R}}

\hfuzz=\maxdimen
\DeclareFixedFont{\Acknowledgment}{OT1}{cmr}{bx}{n}{14pt}
\begin{document}
\title{Circle patterns on surfaces of finite topological type}
\author{Huabin Ge, Bobo Hua, Ze Zhou}
\date{}

\maketitle

\begin{abstract}
This paper investigates circle patterns with obtuse exterior intersection angles on surfaces of finite topological type. We characterise the images of the curvature maps and establish several equivalent conditions regarding long time behaviors of Chow-Luo's combinatorial Ricci flows for these patterns. As consequences, several generalizations of circle pattern theorem are obtained. Moreover, our approach suggests a computational method to find the desired circle patterns.

\medskip
\noindent{\bf Mathematics Subject Classifications (2000):} 52C26, 52C25.

\end{abstract}

%\tableofcontents

\setcounter{section}{-1}

\section{Introduction}
\subsection{Background}The patterns of circles were introduced as useful tools to study the geometry and topology of 3-manifolds by  Thurston \cite[Chap. 13]{T1}. He also posed a conjecture regarding the convergence of infinitesimal hexagonal tangent circle patterns to conformal mappings \cite{T2}, which was proved by Rodin-Sullivan \cite{RS}. From then on, circle patterns  have played significant roles in various problems in combinatorics \cite{Schr1, Schr2,Liu}, discrete and computational geometry \cite{Stephenson, Dai}, minimal surfaces \cite{Bob1}, and many others.

Let $\mathcal T$ be a triangulation of a compact oriented surface $S$ (possibly with boundary) of finite topological type. Suppose that $\mu$ is a constant curvature metric on $S$. A circle pattern $\mathcal P$ on $(S,\mu)$ is a collection of oriented circles. And $\mathcal P$ is called $\mathcal T$-type, if there exists a geodesic triangulation $\mathcal T(\mu)$ of $(S,\mu)$ with the following properties: $(i)$ $\mathcal T(\mu)$ is isotopic to $\mathcal T$; $(ii)$ the vertices of $\mathcal T(\mu)$ coincide with the centers of the circles in $\mathcal P$. Assume that $V,E,F$ are the sets of vertices, edges and triangles of $\mathcal T$.  In this paper, we mainly focus on these $\mathcal T$-type circle patterns $\mathcal P=\{C_v: v\in V\}$
  such that $C_{u}$ and $C_{w}$ intersect with each other whenever there exists an edge between $u$ and $w$. Under this condition we have the exterior intersection angle $\Theta(e)\in[0,\pi)$ for every $e\in E$.  One refers to Stephenson's monograph
\cite{Stephenson} for more details on circle patterns.

Given a function $\Theta: E \to[0,\pi)$ defined on the edge set of $\mathcal T$, let us consider the following question: Does there exist a $\mathcal T$-type circle pattern whose exterior intersection angle function is given by $\Theta $? And if it does, to what extent is the circle pattern unique? Under the condition that $S$ has empty boundary and $0\leq \Theta\leq \pi/2$, a celebrated answer to this question is the following circle pattern theorem due to Thurston \cite[Chap. 13]{T1}.
\begin{theorem}[Thurston]\label{T-0-1}
Let $\mathcal T$ be a triangulation of an oriented closed surface $S$ of genus $g>0$. Suppose that $\Theta: E \to [0,\pi/2]$ is a function satisfying the following conditions:
\begin{itemize}
 \item[$(i)$] If the edges $e_1,e_2,e_3$ form a null-homotopic closed curve in $S$, and if $
 \sum_{l=1}^3 \Theta(e_l)\geq \pi$, then these edges form the boundary of a triangle of $\mathcal T$;
\item[$(ii)$] If the edges $e_1,e_2,e_3,e_4$ form a null-homotopic closed curve in $S$ and if $\sum_{l=1}^4 \Theta(e_i)=2\pi$, then these edges form the boundary of the union of two adjacent triangles.
\end{itemize}
Then there exists a constant curvature (equal to $0$ for $g=1$ and equal to $-1$ for $g>1$) metric $\mu$ on $S$ such that $(S,\mu)$ supports a $\mathcal T$-type circle pattern $\mathcal P$ with the exterior intersection angles given by $\Theta$. Moreover,  the pair $(\mu,\mathcal P)$ is unique up to isometries if $g>1$, and up to similarities if $g=1$.
\end{theorem}

One may ask whether we can relax the requirement of non-obtuse angles in the above theorem. Recently, it was partially resolved by Zhou \cite{Zhou-arxiv}. Let us introduce some terms. A closed (not necessarily simple) curve $\gamma$ in $S$ is called \textbf{pseudo-Jordan},  if $S\setminus \gamma$ contains a simply-connected component whose boundary is $\gamma$. For a \textbf{pseudo-Jordan} curve $\gamma$, an \textbf{enclosing set} $A$ of $\gamma$  consists of all vertices covered by
$\mathbb K$, where $\mathbb K$ is a simply-connected component of $S\setminus\gamma $  such that $\partial \mathbb K=\gamma$. {And we say $\gamma$ is \textbf{non-vacant} if one of its \textbf{enclosing sets} is non-empty}. Zhou's result \cite{Zhou-arxiv} is stated as follows.
\begin{theorem}\label{T-0-2}
Let $\mathcal T$ be a triangulation of an oriented closed surface $S$ of genus $g>1$. Suppose that $\Theta:E\to [0,\pi)$ is a function satisfying the following conditions:
\begin{itemize}
\item[$\mathbf{(C1)}$]If the edges $e_1,e_2,e_3$ form the boundary of a triangle of $\mathcal T$, then $ I(e_1)+I(e_2)I(e_3)\geq 0$, $I(e_2)+I(e_3)I(e_1)\geq 0$, $I(e_3)+I(e_1)I(e_2) \geq 0$, where $I(e_i)=\cos\Theta(e_i)$ for $i=1,2,3$.
\item[$\mathbf{(C2)}$]If the edges $e_1,e_2,\cdots,e_s$ form a \textbf{non-vacant} \textbf{pseudo-Jordan} curve in $S$, then $\sum_{l=1}^s\Theta(e_l)<(s-2)\pi$.
\end{itemize}
Then there exists a hyperbolic metric $\mu$ on $S$ such that  $(S,\mu)$  supports a $\mathcal T$-type  circle pattern $\mathcal P$ with the exterior intersection angles given by $\Theta$. Moreover, the pair $(\mu,\mathcal P)$ is unique up to isometries.
\end{theorem}

We mention that if  $\Theta: E\to [0,\pi)$ satisfies  $\sum_{l=1}^3\Theta(e_l)\leq \pi$ for any three edges $e_1,e_2,e_3$ forming the boundary of a triangle of $\mathcal T$, then \textbf{(C1)} holds (see Remark \ref{R-1-2} in Section \ref{Sec1}). This gives a family of examples possessing possibly obtuse exterior intersection angles for the above theorem. Thurston  proved his theorem via continuity method. The main tool of Zhou's approach to Theorem \ref{T-0-2} is topological degree theory.

\subsection{Main results}

It is of interest to study circle patterns on surfaces with boundary. Suppose that $S$ is of topological type $(g,n)$, i.e. $S$ is a compact oriented surface of genus $g$  whose boundary consists of $n$ disjoint simple closed curves. Let $\chi(\cdot)$ denote the Euler characteristic of a surface, which yields that $\chi(S)=2-2g-n$. A major purpose of this paper is to consider the analogous results to Theorem \ref{T-0-2} in case that $n\geq 0$.

An arc in $S$ is called \textbf{semi pseudo-Jordan} if there exists an open arc belonging to the boundary such that their union is a \textbf{pseudo-Jordan} curve in $S$. For a \textbf{semi pseudo-Jordan} arc $\gamma$, a \textbf{semi enclosing set} $W$ of $\gamma$ consists of all vertices covered by $\lambda\cup\mathbb K$, where $\lambda$ is an open arc in $\partial S$ and $\mathbb K$ is a simply-connected component of $ S\setminus\big(\gamma\cup\lambda\big)$ such that $\partial \mathbb K=\gamma\cup\lambda$. And $\gamma$ is called \textbf{non-vacant} if one of its \textbf{semi enclosing sets} is non-empty. We have the following result.

\begin{theorem}\label{T-0-3}
 Let $\mathcal T$ be a triangulation of a surface $S$ of topological type $(g,n)$ such that $\chi(S)\leq 0$. Suppose that $\Theta:E\to [0,\pi)$ satisfies $\mathbf{(C1)}$, $\mathbf{(C2)}$ and the following condition:
\begin{itemize}
\item[$\mathbf{(C3)}$] If the edges $e_1,e_2,\cdots,e_s$ form a \textbf{non-vacant} \textbf{semi pseudo-Jordan} arc in $S$, then
    $\sum_{l=1}^s\Theta(e_l)<(s-1)\pi$.
\end{itemize}
Then there exists a constant curvature (equal to $0$ for $\chi(S)=0$ and equal to $-1$ for $\chi(S)<0$) metric $\mu$ on $S$ such that $(S,\mu)$ has totally geodesic boundary and supports a $\mathcal T$-type circle pattern $\mathcal P$ with the exterior intersection angles given by $\Theta$. Moreover, the pair $(\mu,\mathcal P)$ is unique up to isometries if $\chi(S)<0$, and up to similarities if $\chi(S)=0$.
\end{theorem}

To obtain the proof, let us introduce Thurston's construction for circle patterns \cite[Chap. 13]{T1}. Recall that $V,E,F$ are the sets of vertices, edges, and triangles of  $\mathcal T$. Assume that $V=\big\{v_1,\cdots,v_{|V|}\big\}$. We start with a radius vector $r=(r_1,\cdots,r_{|V|})\in \mathbb R_{+}^{|V|}$, which assigns each vertex $v_i$ a positive number $r_i$. A radius vector  produces a hyperbolic (or Euclidean) cone metric on  $S$ as follows.

For each combinatorial triangle $\bigtriangleup (v_iv_jv_k)$ of $\mathcal T$, one associates it with a hyperbolic (or Euclidean) triangle formed by centers of three hyperbolic circles/disks of radii $r_{i}, r_{j}, r_{k}$ with exterior intersection angles $\Theta([v_i,v_j]), \Theta([v_j,v_k]),\Theta([v_k,v_i])$. More precisely, let $l_{ij}, l_{jk},l_{ki}$ be the three lengths of this triangle. Then
\[
l_{ij}\ =\ \cosh^{-1}\big(\cosh r_{i}\cosh r_{j}+\sinh r_{i}\sinh r_{j}\cos\Theta([v_i,v_j])\big)
\]
in hyperbolic background geometry, or
\[
l_{ij}\,=\,\sqrt{r_i^2+r_j^2+2r_ir_j\cos\Theta([v_i,v_j])}
\]
in Euclidean background geometry. Similarly, we obtain $l_{jk}$, $l_{ki}$. Under the condition \textbf{(C1)}, for any three positive numbers $r_{i}, r_{j}, r_{k}$, the corresponding $l_{ij}, l_{jk}, l_{ki}$
satisfy the triangle inequalities (see Lemma \ref{L-1-1} in Section \ref{Sec1}). Thus the above procedure works well.

Gluing all these hyperbolic (or Euclidean) triangles along the common edges produces a hyperbolic (or Euclidean) cone metric on $S$ with possible cone singularities at vertices of $\mathcal T$. Let $V_\partial\subset V$ be the set of boundary vertices. For each $v_i\in V$, the vertex curvature $K_i$ is defined as follows:
\[K_i:=\left\{\begin{array}{ll}2\pi-\sigma (v_i), &v_i\in V\setminus V_\partial,\\
\pi-\sigma (v_i), &v_i\in V_\partial.
\end{array}\right.\]
Here $\sigma(v_i)$ denotes the cone angle at $v_i$, which is equal to the sum of inner angles at $v_i$ for all triangles incident to $v_i$. Clearly,  $K_1,\cdots, K_{|V|}$ are smooth functions of $r$. This gives rise to the following curvature map
\[
\begin{aligned}
Th(\cdot):\quad\qquad &\,\mathbb{R}_{+}^{|V|} \qquad
&\,\to \,\qquad \qquad \quad &\,\mathbb{R}^{|V|} \\
\big(r_1,r_2,&\cdots,r_{|V|}\big)&\mapsto\;\,\ \quad\big(K_1, K_2,&\cdots,K_{|V|}\big).\\
  \end{aligned}
  \]

The aim is to show that the origin $(0,0,\cdots,0)$ belongs to the image of the map $Th$. If there exists a radius vector $r^\ast$ such that  $K_i(r^\ast)=0$ for $i=1,2,\cdots,|V|$, then it produces a smooth hyperbolic (or Euclidean) metric on $S$. Drawing the circle centered at $v_i$ of radius $r^{\ast}_i$ for each $v_i$, one  obtains the desired circle pattern realizing $(\mathcal T,\Theta)$.

For any non-empty subset $A$ of $V$,  we denote by $G(A)$ the union of $j$-cells ($j=0,1,2$) of $\mathcal T$ that have at least one vertex in $A$, and by $Lk(A)$ the set of pairs $(e,v)$ of an edge $e$ and a vertex $v$ with the following properties:
 $$(i)\ v\in A;\quad (ii)\ \partial e\cap A=\emptyset;\quad (iii)\ e \ \mathrm{and}\ v \mathrm{\ form\  a\  triangle\ of\ } \mathcal T.$$
The following results give a complete description of the images of the curvature maps.

\begin{theorem}\label{T-0-4}
Assume that $\Theta:E\to [0,\pi)$ satisfies  $\mathbf{(C1)}$. In hyperbolic background geometry, the curvature map is injective. Moreover, the image of this map consists of vectors $(K_1,K_2,\cdots,K_{|V|})$ satisfying
\begin{equation}\label{E-1}
K_i\,<\,\left\{\begin{array}{ll}
2\pi,&v_i\in V\setminus V_\partial,\\
\pi,&v_i\in V_\partial,
\end{array}\right.
\end{equation}
and
\begin{equation}\label{E-2}
\sum\nolimits_{v_i\in A}K_i\,>\,-\sum\nolimits_{(e,v)\in Lk(A)}\big(\pi-\Theta(e)\big)+2\pi\chi(G(A)\setminus \partial S)+\pi \chi(G(A)\cap\partial S)
\end{equation}
for any non-empty subset $A$ of $V$.
\end{theorem}

\begin{theorem}\label{T-0-5}
Assume that $\Theta:E\to [0,\pi)$ satisfies $\mathbf{(C1)}$. In Euclidean background geometry, the curvature map is injective up to scalings. Moreover, the image of this map consists of vectors $(K_1,K_2,\cdots,K_{|V|})$ satisfying $(\ref{E-1})$ and
\begin{equation}\label{E-3}
\sum\nolimits_{v_i\in A}K_i\,\geq \,-\sum\nolimits_{(e,v)\in Lk(A)}\big(\pi-\Theta(e)\big)+2\pi\chi(G(A)\setminus \partial S)+\pi \chi(G(A)\cap\partial S)
\end{equation}
for any non-empty subset $A$ of $V$, where the equality holds if and only if $A=V$.
\end{theorem}

\begin{remark}\label{R-0-6}
Setting $A=V$, then (\ref{E-2}) (resp. (\ref{E-3})) gives the following Gauss-Bonnet inequality (resp. equality)
\begin{equation}\label{E-4}
\sum\nolimits_{i=1}^{|V|}K_i\,>\,2\pi\chi(S) \quad \big(\,\text{resp.}\;\sum\nolimits_{i=1}^{|V|}K_i\,=\,2\pi\chi(S)\,\big).
\end{equation}
\end{remark}

Under proper conditions, a combinatorial argument will show that the origin belongs to the images. A more delicate problem is to search the radius vector which gives the desired circle pattern. For this purpose, we follow  Chow-Luo's work on combinatorial Ricci flows \cite{ChowLuo-jdg}. As a comparison to the topological degree method, this approach has the advantage of providing an algorithm (converging exponentially fast) to find the solution.

Specifically, in hyperbolic background geometry, one considers the ODE system
\begin{equation}
\label{E-5}
\frac{dr_i}{dt}\,=\,-K_i\sinh r_i
\end{equation}
for $i=1,2\cdots,|V|$, with an initial radius vector $r(0)\in \mathbb{R}^{|V|}_{+}$. In Euclidean background geometry, we consider the ODE system
\begin{equation}
\label{E-6}
\frac{dr_i}{dt}\,=\,(K_{av}-K_i)r_i,
\end{equation}
where $K_{av}=2\pi\chi(S)/|V|$.

\begin{theorem}\label{T-0-7}
Suppose that $\Theta:E\to [0,\pi)$ satisfies  $\mathbf{(C1)}$. In hyperbolic background geometry, the solution $r(t)$ to the flow $(\ref{E-5})$ exists for all the time $t\geq0$, and the following properties $H_1$-$H_4$ are equivalent:
\begin{itemize}
\item[$H_1.$] $r(t)$ converges as $t\to +\infty$.
\item[$H_2.$] The origin $(0,\cdots,0)$ belongs to  the image of the curvature map.
\item[$H_3.$] If $A$ is a non-empty subset of $V$, then
\begin{equation*}
0\,>\,-\sum\nolimits_{(e,v)\in Lk(A)}\big(\pi-\Theta(e)\big)+2\pi\chi(G(A)\setminus \partial S)+\pi\chi(G(A)\cap\partial S).
\end{equation*}
\item[$H_4.$] $S$ has negative Euler characteristic and $\Theta$ satisfies the conditions $\mathbf{(C2)}$, $\mathbf{(C3)}$.
\end{itemize}
Moreover, if one of the above properties holds, then the flow \eqref{E-5} converges exponentially fast to a radius vector which produces a smooth hyperbolic metric $\mu$ on $S$, so that $(S,\mu)$ supports a circle pattern $\mathcal P$ realizing $(\mathcal T,\Theta)$.
 \end{theorem}

\begin{theorem}\label{T-0-8}
Suppose that $\Theta:E\to [0,\pi)$ satisfies  $\mathbf{(C1)}$. In Euclidean background geometry, the solution $r(t)$ to the flow (\ref{E-6}) exists for all the time $t \geq 0$, and the following properties $E_1$-$E_3$ are  equivalent:
\begin{itemize}
\item[$E_1.$] $r(t)$ converges as $t\to +\infty$.
\item[$E_2.$] The vector $(\,K_{av},\cdots,K_{av}\,)$ belongs to  the image of the curvature map.
\item[$E_3.$] If $A$ is a proper non-empty subset of $V$, then
\begin{equation*}
2\pi\chi(S)|A|/|V|\,>\,-\sum\nolimits_{(e,v)\in Lk(A)}\big(\pi-\Theta(e)\big)+2\pi\chi(G(A)\setminus \partial S)+\pi\chi(G(A)\cap\partial S).
\end{equation*}
\end{itemize}
{Moreover, if one of the above properties holds, then the flow \eqref{E-6} converges exponentially fast to a radius vector which produces an Euclidean cone metric $\mu$ on $S$, so that $(S,\mu)$ has vertex curvatures all equal to $K_{av}$ and supports a circle pattern $\mathcal P$ realizing $(\mathcal T,\Theta)$}.
 \end{theorem}

With proper modifications to the flows (\ref{E-5}) and (\ref{E-6}), one can search circle patterns with other prescribed vertex curvatures. The following is a result generalizing Theorem \ref{T-0-3}.

\begin{theorem}\label{T-0-9}
 Let $\mathcal T$ be a triangulation of a surface $S$ of topological type $(g,n)$. Suppose that $\Theta:E\to [0,\pi)$ and $\varphi:V_\partial\to [0,\pi)$ are two functions such that  $\mathbf{(C1)}$, $\mathbf{(C2)}$ and  the following conditions are satisfied:
 \begin{itemize}
 \item[$\mathbf{(BV1)}$] The Gauss-Bonnet inequality (resp. equality) holds:
\[
\sum\nolimits_{v\in V_\partial}\varphi(v) \,>\,2\pi\chi(S) \quad \big(\;\text{resp.}\;\sum\nolimits_{v\in V_\partial}\varphi(v)\,=\,2\pi\chi(S)\;\big).
\]
\item[$\mathbf{(BV2)}$] If the edges $e_1,\cdots,e_s$ form a \textbf{non-vacant} \textbf{semi pseudo-Jordan} arc $\gamma$ in $S$, then for any \textbf{non-empty semi enclosing set} $W$ of $\gamma$
 \[
\sum\nolimits_{v\in W\cap V_\partial}\varphi(v)+
\sum\nolimits_{l=1}^s\big(\pi-\Theta(e_l)\big)\,>\,\pi.
 \]
\end{itemize}
Then there exists a hyperbolic (resp. Euclidean) metric on $S$ so that $(S,\mu)$, whose boundary consists of $n$ disjoint simple piecewise-geodesic closed curves with turning angles assigned by $\varphi$, supports a $\mathcal T$-type circle pattern
$\mathcal P$ with the exterior intersection angles given by $\Theta$.
 Moreover, the pair $(\mu,\mathcal P)$ is unique up to isometries (resp. similarities).
\end{theorem}

\begin{figure}[htbp]\centering
\includegraphics[width=0.55\textwidth]{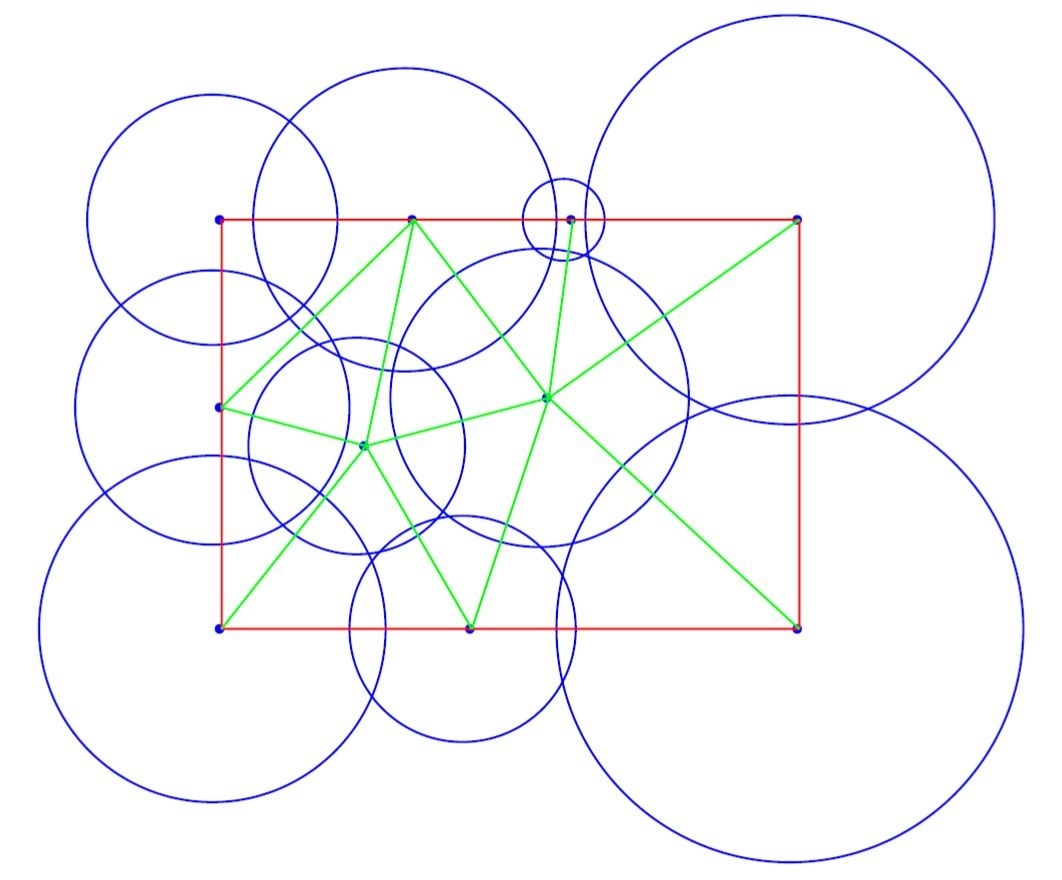}
\caption{A circle pattern on a rectangle}
\end{figure}

Theorem \ref{T-0-9} can be applied to study planar circle patterns.
As an example, we have the following result regarding circle patterns on rectangles. More relevant results can be found in Section \ref{Further}.
\begin{corollary}\label{C-0-10}
 Let $\mathcal T$ be a triangulation of a quadrangle. Suppose that  $\Theta:E\to [0,\pi)$ is a function satisfying $\mathbf{(C1)}$, $\mathbf{(C2)}$ and the following conditions:
 \begin{itemize}
 \item[$\mathbf{(R1)}$] If the edges $e_1,\cdots,e_s$ form a \textbf{semi pseudo-Jordan} arc joining a pair of opposite corner vertices, then $\sum\nolimits_{l=1}^s \Theta(e_l))<(s-1/2)\pi$.
 \item[$\mathbf{(R2)}$] If the edges $e_1,\cdots,e_s$ form a \textbf{semi pseudo-Jordan} arc which does not join a pair of opposite corner vertices, then $\sum\nolimits_{l=1}^s \Theta(e_l))<(s-1)\pi$.
\end{itemize}
Then there exists a rectangle $R$ which supports a circle pattern $\mathcal P$ with exterior intersection angles given by $\Theta$. Moreover, the pair $(R,\mathcal P)$ is unique up to similarities.
\end{corollary}

The paper is organized as follows: In next section, we introduce basic properties of three-circle configurations. In Section~\ref{sec:hyperbolic}, we study circle patterns in hyperbolic background geometry, and prove Theorem~\ref{T-0-4} and Theorem~\ref{T-0-7}. In Section~\ref{sec:euclidean}, we deal with circle patterns in Euclidean background geometry, and prove Theorem~\ref{T-0-5} and Theorem~\ref{T-0-8}. In Section \ref{Further}, we focus on circle patterns with prescribed vertex curvatures and prove Theorem \ref{T-0-9}. As consequences, several results on planar circle patterns are established. The last section is an appendix on some combinatorial facts.

\medskip
{Throughout this paper, we denote by $|\cdot|$ the cardinality of a set}.

\section{Preliminaries on three-circle configurations}\label{Sec1}

In this section we establish several lemmas on three-circle configurations. It should be pointed out that the non-obtuse versions of these results have been appeared in  \cite{T1,Marden-Rodin,ChowLuo-jdg}. To simplify notations, for three angles $\Theta_i, \Theta_j, \Theta_k\in [0,\pi)$,  set
\[
\xi_{i}\,=\,\cos\Theta_i+\cos\Theta_j\cos\Theta_k,
\]
and $\xi_{j}, \xi_{k}$ similarly.

\begin{lemma}\label{L-1-1}
Suppose that $\Theta_i, \Theta_j, \Theta_k \in [0,\pi) $ are three angles satisfying
\begin{equation*}
\xi_{i}\,\geq\, 0,\;\;\xi_{j}\,\geq\, 0,\;\;\xi_{k}\,\geq\, 0.
\end{equation*}
For any three positive numbers $r_i,r_j,r_k$, there exists a configuration of three intersecting disks in both Euclidean and hyperbolic geometries, unique up to isometry, having radii $r_i,r_j,r_k$  and meeting in exterior intersection angles $\Theta_i,\Theta_j,\Theta_k$.
\end{lemma}

\begin{figure}[htbp]\centering
\includegraphics[width=0.5\textwidth]{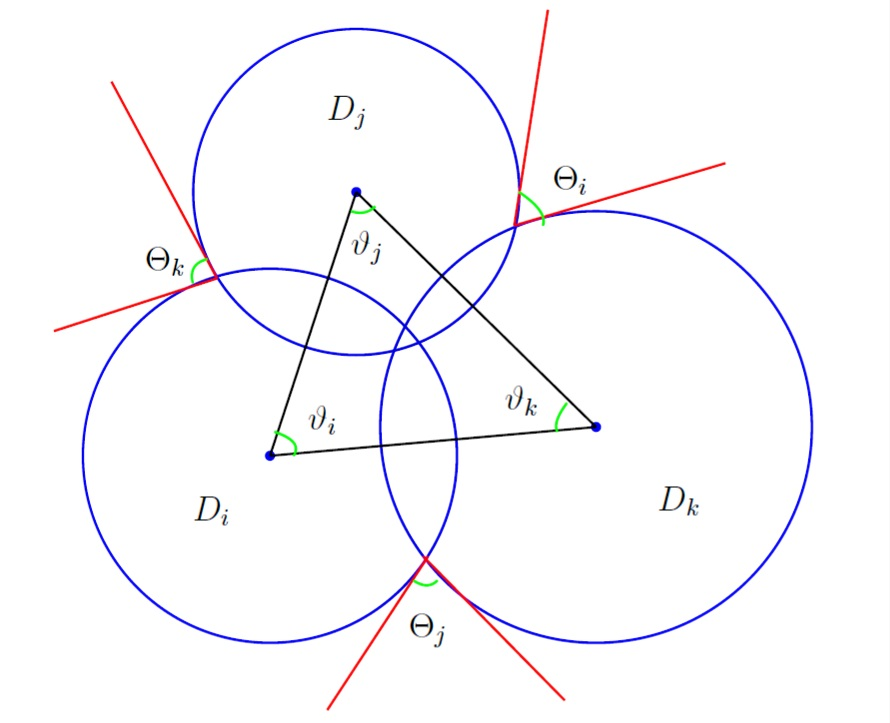}
\caption{A three-circle configuration}\label{fig2}
\end{figure}

\begin{proof}
For the Euclidean case, set
\[
l_i\,=\,\sqrt{r_j^2+r_k^2+2\cos\Theta_ir_jr_k},
\]
and $l_j,l_k$ similarly. It suffices to check that $l_i,l_j,l_k$ satisfy the triangle inequalities. Without loss of generality, we will prove the inequality
\[
|l_i-l_j|\,<\,l_k.
\]
That is,
\[
\Big|\sqrt{r_j^2+r_k^2+2\cos\Theta_ir_jr_k}-\sqrt{r_k^2+r_i^2+2\cos\Theta_jr_kr_i} \,\Big|\,<\,\sqrt{r_i^2+r_j^2+2\cos\Theta_kr_ir_j}\,,
\]
which is equivalent to the following inequality
\begin{equation*}
\sin^2\Theta_ir_j^2r_k^2+\sin^2\Theta_jr_k^2 r_i^2+\sin^2 \Theta_kr^2_ir^2_j
+2\xi_{i}r_i^2r_jr_k+2\xi_{j}r_j^2r_kr_i+2\xi_{k}r_k^2 r_ir_j\,>\,0.
\end{equation*}
Under the conditions that $\xi_{i}\geq 0$, $\xi_{j}\geq 0$, $\xi_{k}\geq 0$, the above inequality holds. Thus the statement follows.

The hyperbolic case has been established in  \cite{Zhou-arxiv}. We include it here for the sake of completeness. Let $l_i>0$ such that
\[
\cosh l_i\ =\ \big(\cosh r_{j}\cosh r_{k}+\sinh r_{j}\sinh r_{k}\cos\Theta_i\big)
\]
Define $l_j, l_k$ similarly. It suffices to check that $l_i,l_j,l_k$ satisfy the triangle inequalities. In other words,
\begin{equation}\label{E-7}
\cosh(l_i+l_j)\ >\ \cosh l_k,
\end{equation}
and
\begin{equation}\label{E-8}
\cosh(l_i-l_j)\ <\ \cosh l_k.
\end{equation}
Combining  (\ref{E-7}) and (\ref{E-8}), we need to show
\begin{equation}\label{E-9}
(\cosh l_i\cosh l_j-\cosh l_k)^2\ <\ \sinh^2 l_i\sinh^2 l_j.
\end{equation}
To simplify the notations, for $\eta=i,j,k$, set
\[
a_\eta\ =\ \cosh r_\eta,\;x_\eta\ =\ \sinh r_\eta,\;I_\eta\ =\ \cos\Theta_\eta.
\]
Then
\begin{equation}\label{E-10}
\cosh l_i\cosh l_j-\cosh l_k\ =\ (a_ia_j+I_iI_jx_ix_j)x_k^2+(I_ia_ix_j+ I_ja_jx_i) a_kx_k-I_kx_ix_j.
\end{equation}
Substituting (\ref{E-10}) into (\ref{E-9}), it is equivalent to proving
\begin{equation*}
\begin{split}
&\sin^2\Theta_ix_j^2x^2_k+\sin^2\Theta_jx^2_k x_i^2+\sin^2\Theta_kx^2_ix^2_j+\big(2+2I_iI_jI_k\big)x_i^2x_j^2x_k^2\\
&\quad+2\xi_{i}a_ja_kx_jx_kx_i^2+2\xi_{j}a_ka_ix_kx_ix_j^2+2\xi_{k}a_ia_jx_ix_jx_k^2\;>\;0.
\end{split}
\end{equation*}
Note that
\[
\xi_{i}\ \geq\  0, \;\xi_{j}\  \geq\ 0,\;\xi_{k}\ \geq\ 0,
\]
and
\[
2+2I_iI_jI_k\ =\ 2+2\cos\Theta_i\cos\Theta_j\cos\Theta_k\ >\ 0.
\]
We thus prove the above inequality and finish the proof.
\end{proof}

\begin{remark}\label{R-1-2}
If $\Theta_i,\Theta_j,\Theta_k\in[0,\pi)$ satisfy $\Theta_i+\Theta_j+\Theta_k\leq \pi$, then
\[
\xi_{i}\,=\,\cos\Theta_i+\cos\Theta_j\cos\Theta_k\,\geq\,2\cos\frac{\Theta_i+\Theta_j+\Theta_k}{2}\cos\frac{\Theta_i-\Theta_j-\Theta_k}{2}\\
\,\geq\, 0.
\]
Similarly, we show that $\xi_{j}\geq 0$ and $\xi_{k}\geq 0$.
\end{remark}

As in Figure~\ref{fig2}, let $\bigtriangleup_{ijk}$ be the triangle whose vertices are the centers of three intersecting circles. Let $\vartheta_i,\vartheta_j,\vartheta_k$ denote the corresponding inner angles at the centers.

\begin{lemma}\label{L-1-3}
Let $\Theta_i, \Theta_j, \Theta_k \in [0,\pi) $ satisfy  the conditions of Lemma \ref{L-1-1}.
\begin{itemize}
\item[$(i)$]In Euclidean geometry,
\[
\frac{\partial\vartheta_i}{\partial r_i}\,<\,0,\quad r_j\frac{\partial\vartheta_i}{\partial r_j}\,=\,r_i\frac{\partial\vartheta_j}{\partial r_i}\,\geq\,0.
\]
\item[$(ii)$]In hyperbolic geometry,
\[
\frac{\partial\vartheta_i}{\partial r_i}\,<\,0,\quad \sinh r_j\frac{\partial\vartheta_i}{\partial r_j}\,=\,\sinh r_i\frac{\partial\vartheta_j}{\partial r_i}\,\geq\,0,\quad \frac{\partial \mathrm{Area}(\bigtriangleup_{ijk})}{\partial r_i} \,>\,0,
\]
where $\mathrm{Area}(\bigtriangleup_{ijk})$ denotes the area of $\bigtriangleup_{ijk}$.
\end{itemize}
\end{lemma}

\begin{proof}
We follow the arguments in \cite{Zhou-arxiv} and \cite{Xu-arxiv} and use the same notations in the proof of Lemma \ref{L-1-1}.
In Euclidean geometry, it follows from the cosine law that
\[
\cos\vartheta_i\,=\,\frac{l_j^2+l_k^2-l_i^2}{2l_jl_k}.
\]
Taking the derivative with respect to $l_i$, we have
\[
\frac{\partial\vartheta_i}{\partial l_i}\,=\,\frac{l_i}{ \Gamma_{ijk}},
\]
where $\Gamma_{ijk}=\sin\vartheta_i l_jl_k$. By the sine law of Euclidean triangles, it is easy to see
\[
\Gamma_{ijk}\,=\,\Gamma_{jki}\,=\,\Gamma_{kij}\,=:\,\Gamma.
\]
Similarly, we obtain
\[
\frac{\partial\vartheta_i}{\partial l_k}\,=\,-\frac{l_i\cos\vartheta_j}{\Gamma_{ijk}}\,=\,-\frac{l_i\cos\vartheta_j}{\Gamma}.
\]
Note that $
l_i^2=r_j^2+r_k^2+2r_jr_kI_i$,
 we have
 \[
\frac{\partial l_i}{\partial r_j}\,=\,\frac{r_j+r_kI_i}{l_i}.
 \]
 Similarly,
 \[
\frac{\partial l_k}{\partial r_j}\,=\,\frac{r_j+r_iI_k}{l_k}.
 \]
 A direct calculation gives
\begin{equation*}
\begin{aligned}
r_j\frac{\partial\vartheta_i}{\partial r_j}\,=\,&\,r_j\,\Bigg(\,\frac{\partial\vartheta_i}{\partial l_i}\frac{\partial l_i}{\partial r_j}+\frac{\partial\vartheta_i}{\partial l_k}\frac{\partial l_k}{\partial r_j}\,\Bigg)\\
\,=\,&\,\frac{\sin^2\Theta_k r_i^2r_j^2+(\xi_{i}r_i+\xi_{j}r_j)r_ir_jr_k}{\Gamma l_k^{2}}.
\end{aligned}
\end{equation*}
It is easy to see
\[
r_j\frac{\partial\vartheta_i}{\partial r_j}\,=\,r_i\frac{\partial\vartheta_j}{\partial r_i}\,\geq\,0.
\]
Moreover, $ \partial\vartheta_i/\partial r_j=0$ holds if and only if $\Theta_k=0$ and $\Theta_i+\Theta_j=\pi$. Consequently,
$$r_i\frac{\partial\vartheta_i}{\partial r_i}\,=\,-r_i\,\Bigg(\frac{\partial\vartheta_j}{\partial r_i}+\frac{\partial\vartheta_k}{\partial r_i}\Bigg)
\,=\,-\Bigg(\,r_j\frac{\partial\vartheta_i}{\partial r_j}+r_k\frac{\partial\vartheta_i}{\partial r_k}\,\Bigg)\,\leq\,0.$$
The equality holds if and only if $\Theta_i=\pi$, $\Theta_j=\Theta_k=0$, which contradicts to the assumption that $\Theta_i\in[0,\pi)$. We thus finish the proof of the Euclidean case.

In hyperbolic geometry, due to the cosine and sine laws of hyperbolic triangles, one obtains
\[
\frac{\partial\vartheta_i}{\partial l_i}\,=\,\frac{\sinh l_i}{\Upsilon},\,\quad\quad \frac{\partial\vartheta_i}{\partial l_k}\,=\,-\frac{\sinh l_i\cos\vartheta_j}{\Upsilon},
\]
 where $\Upsilon=\sin\vartheta_i\sinh l_j\sinh l_k$. Moreover, note that
\[
\cosh l_i\,=\,\cosh r_j\cosh r_k+\cos\Theta_i\sinh r_j\sinh r_k\ =\ a_ja_k+I_ix_jx_k.
\]
It follows that
\begin{equation*}
\frac{\partial l_i}{\partial r_j}\ =\ \frac{a_k x_j+I_ia_jx_k}{\sinh l_i}.
\end{equation*}
Similarly,
\[
\frac{\partial l_k}{\partial r_j}\,=\,\frac{a_i x_j+I_ka_jx_i}{\sinh l_k}.
\]
We have
\begin{equation*}
\begin{aligned}
\sinh r_j\frac{\partial\vartheta_i}{\partial r_j}
\,=&\,\sinh r_j\,\Bigg(\,\frac{\partial\vartheta_i}{\partial l_i}\frac{\partial l_i}{\partial r_j}+\frac{\partial\vartheta_i}{\partial l_k}\frac{\partial l_k}{\partial r_j}\,\Bigg)\\
=&\, \frac{\sin^2 \Theta_k a_k^2x_i^2x_j^2+(\xi_{i}a_jx_i+\xi_{j}a_ix_j)x_ix_jx_k}{\Upsilon\sinh^2 l_k}.
\end{aligned}
\end{equation*}
which implies
\[
\sinh r_j\frac{\partial\vartheta_i}{\partial r_j}\,=\,\sinh r_i\frac{\partial\vartheta_j}{\partial r_i}\,\geq\,0.
\]
Similarly, $\partial\vartheta_i/\partial r_j=0$ holds if and only if $\Theta_k=0$ and $\Theta_i+\Theta_j=\pi$.

Fix $r_j$, $r_k$, and let $r_i$ vary. Then $l_i$ stays constant. Because $\partial\vartheta_j/\partial r_i\geq 0$ and $\partial\vartheta_k/\partial r_i\geq 0$, that means the other two edges of the triangle $\triangle_{ijk}$ move outwards or remain unchanged as $r_i$ increases. Furthermore, it is easy to see that these two edges can not stay unchanged simultaneously. Hence the area $\mathrm{Area}(\bigtriangleup_{ijk})$ is a strictly increasing function of $r_i$. Namely,
\[
\frac{\partial \mathrm{Area}(\bigtriangleup_{ijk})}{\partial r_i}\,>\,0.
\]
Due to the Gauss-Bonnet formula, we have
\[
\frac{\partial\vartheta_i}{\partial r_i}\,=\,-\frac{\partial \mathrm{Area}(\triangle_{ijk})}{\partial r_i}-\frac{\partial\vartheta_j}{\partial r_i}-\frac{\partial\vartheta_k}{\partial r_i}\,<\,0.
\]
\end{proof}

\begin{lemma}\label{L-1-4}
In Euclidean geometry,
\[
0\,<\,\vartheta_i\,<\,\pi-\Theta_i, \quad\vartheta_i+\vartheta_j+\vartheta_k\,=\,\pi.
\]
In hyperbolic geometry,
\begin{equation*}
0\,<\,\vartheta_i\,<\,\pi-\Theta_i,\quad \vartheta_i+\vartheta_j+\vartheta_k\,<\,\pi.
\end{equation*}
\end{lemma}

The proof is simple and hence we omit the details.

\begin{lemma}\label{L-1-5}
Given three positive numbers $a,b,c$, in both Euclidean and hyperbolic geometries, we have
\begin{equation*}
\begin{aligned}
&\lim_{(r_i,r_j,r_k)\to(0,a,b)}\vartheta_i\;=\;\pi-\Theta_i, \\
&\lim_{(r_i,r_j,r_k)\to(0,0,c)}\vartheta_i+\vartheta_j\;=\;\pi, \\
&\lim_{(r_i,r_j,r_k)\to(0,0,0)}\vartheta_i+\vartheta_j+\vartheta_k\;=\;\pi. \\
\end{aligned}
\end{equation*}
\end{lemma}

\begin{proof}
{In Euclidean geometry, the first assertion is straightforward. To prove the second assertion, it suffices to show that $\vartheta_k\to 0$ as $(r_i,r_j,r_k)\to(0,0,c)$. By taking the limit in the following equality
\[
\cos\vartheta_k\,=\,\frac{r_{k}^2+r_{j}r_{k}I_{i}+r_{i}r_{k}I_{j}-r_{i}r_{j}I_{k}}{\sqrt{r_{j}^2+r_{k}^2+2r_{j}r_{k}I_{i}}
\sqrt{r_{i}^2+r_{k}^2+2r_{i}r_{k}I_{j}}},
\]
it is easy to check that the statement holds. The third assertion is trivial.}

In hyperbolic geometry, note that
\[
\cos\vartheta_i\,=\,\frac{\cosh l_j\cosh l_k-\cosh l_i}{\sinh l_j\sinh l_k}.
\]
As $(r_i,r_j,r_k)\rightarrow(0,a,b)$, we have $l_j\to b$, $l_k\to a$ and
\[
\cosh l_i\,\to\, \cosh a\cosh b+\cos \Theta_i\sinh a\sinh b.
\]
Hence
\[
\cos\vartheta_i\,\to\, -\cos \Theta_i,
\]
which implies the first assertion. To prove the second one, a routine computation shows
\begin{equation*}
\begin{aligned}
0\,<\,\cos\vartheta_i+\cos\vartheta_j&=\,\frac{\sinh (l_i+l_j)\big[\cosh l_k-\cosh(l_i-l_j)\big]}{\sinh l_i\sinh l_j\sinh l_k}\\
&\leq\,\frac{\sinh (l_i+l_j)(\cosh l_k-1)}{\sinh l_i\sinh l_j\sinh l_k}\\
&=\,\frac{\sinh (l_i+l_j)\sinh(l_k/2)}{\sinh l_i\sinh l_j\cosh(l_k/2)}.
\end{aligned}
\end{equation*}
As $(r_i,r_j,r_k)\rightarrow(0,0,c)$, we have $l_i\rightarrow c$, $l_j\rightarrow c$ and $l_k\rightarrow0$. Therefore,
\[
\cos\vartheta_i+\cos\vartheta_j\,\rightarrow\,0,
\]
which implies $\vartheta_i+\vartheta_j\rightarrow\pi$. For the third assertion, observe that all the three lengths $l_i,l_j,l_k$ tend to zero as $(r_i,r_j,r_k)\rightarrow(0,0,0)$. It follows that $\mathrm{Area}(\bigtriangleup_{ijk})$ tends to zero. Due to Gauss-Bonnet formula, the statement follows.
\end{proof}

\begin{lemma}\label{L-1-6}
In hyperbolic geometry, given $\Theta_i,\Theta_j,\Theta_k\in [0,\pi)$, for any $\epsilon>0$, there exists a positive number $L$ such that for any positive $r_i,r_j,r_k$ satisfying $r_i>L,$
$$\vartheta_i\,<\,\epsilon.$$
As a consequence, we have
\begin{equation*}
\displaystyle{\lim_{r_i\to+\infty}\vartheta_i\;=\;0.}
\end{equation*}
\end{lemma}
\begin{proof}
The non-obtuse case was observed by Chow-Luo \cite{ChowLuo-jdg}. The following proof is similar to those in \cite{GeXu-imrn,GeJiang-jfa2,Zhou-arxiv}.
By the hyperbolic cosine law, we have
\begin{equation*}
\cos \vartheta_i
\,=\,\frac{\cosh(l_{j}+l_{k})+\cosh(l_{j}-l_{k})-2\cosh l_{i}}{\cosh(l_{j}+l_{k})-\cosh(l_{j}-l_{k})}
\,=\,\frac{1+\delta-2\rho}{1-\delta},
\end{equation*}
where
\[
\delta\,=\,\frac{\cosh(l_{j}-l_{k})}{\cosh(l_{j}+l_{k})},\quad \rho\,=\,\frac{\cosh l_{i}}{\cosh(l_{j}+l_{k})}.
\]
We shall prove $\delta,\, \rho\rightarrow 0$ uniformly as $r_i\rightarrow +\infty$. For $\eta=i,j,k$, setting $c_\eta=\min\{\cos\Theta_\eta, 0\}$, we have
$0<1+c_\eta\leq 1$. Then
\begin{equation*}
\begin{aligned}
\cosh l_k\,&=\,\cosh r_i\cosh r_j+\sinh r_i\sinh r_j\cos\Theta_k\\
&\geq\,\cosh r_i\cosh r_j+c_k\cosh r_i\cosh r_j\\
&\geq\,(1+c_k)\cosh r_i\cosh r_j\\
&\geq\,(1+c_k)\cosh r_i.
\end{aligned}
\end{equation*}
Similarly,
\[
\cosh l_j\,\geq\,(1+c_j)\cosh r_i\cosh r_k\geq(1+c_j)\cosh r_i.
\]
We obtain
\[
0\,<\,\delta\,<\,\frac{\max\{\cosh l_{j},\cosh l_{k}\}}{\cosh (l_{j}+l_{k})}\,<\,\frac{1}{\min\{\cosh l_{j},\cosh l_{k}\}}\,\leq\,\frac{1}{(1+c_j)(1+c_k)\cosh r_i},
\]
we have $\delta\rightarrow 0$ uniformly as $r_i\rightarrow +\infty$. Furthermore, note that
\[
\cosh l_i\,\leq\, 2\cosh r_j\cosh r_k.
\]
We get
\[
0\,<\,\rho\,\leq\,\frac{\cosh l_i}{\cosh l_j\cosh l_k}\,\leq\,\frac{2}{(1+c_j)(1+c_k)\cosh^2r_i}.
\]
Then $\rho\rightarrow0$ uniformly as $r_i\rightarrow +\infty$.
\end{proof}

\section{Hyperbolic background geometry}\label{sec:hyperbolic}

\subsection{Curvature map and energy function}
Recall that $\mathcal T$ is a triangulation of $S$ with the sets of vertices, edges and triangles $V,E,F$. Given a non-empty subset $A$ of $V$, let $F(A)\subset F$ be the set of triangles that have at least one vertex in $A$. For a vertex $v_i\in V$ and a triangle $\bigtriangleup\in F$ incident to $v_i$, we denote $\vartheta_i^{\triangle}$ by the inner angle at $v_i$ of the triangle $\bigtriangleup$. Therefore, the cone angle $\sigma(v_i)$ at $v_i$ can be expressed as
\[
\sigma(v_i)\,=\,\sum\nolimits_{\triangle\in F(\{v_i\})} \vartheta_i^{\triangle}.
\]
For $i=1,2,\cdots,|V|$, using the change of variables $u_i=\ln \tanh (r_i/2)$, one regards the vertex curvatures as smooth functions of  $u=(u_1,\cdots,u_{|V|})$. This induces the smooth curvature map in terms of $u$:
\[
\begin{aligned}
Th(\cdot):\quad\qquad &\,\mathbb{R}_{-}^{|V|} \qquad
&\,\to \,\qquad \qquad \quad &\,\mathbb{R}^{|V|} \\
\big(u_1,u_2,&\cdots,\,u_{|V|}\big)&\mapsto\;\,\quad\big(K_1,\, K_2,&\cdots,\,K_{|V|}\big).\\
 \end{aligned}
\]

\begin{lemma}\label{L-2-1}
The Jacobian matrix of $Th$ in terms of $u$ is symmetric and positive definite.
\end{lemma}

\begin{proof}
If $v_i,v_j$ is a pair of non-adjacent vertices, then
\begin{equation*}
\frac{\partial K_i}{\partial u_j}\,=\,\frac{\partial K_j}{\partial u_i}\,=\,0.
\end{equation*}
Otherwise, suppose that $[v_i,v_j]=e\in E$. In case that $e$ is not a boundary edge, then there exist two triangles $\triangle_1,\triangle_2$ adjacent to $e$. By Lemma \ref{L-1-3}, a simple computation gives
\begin{equation*}
\frac{\partial K_i}{\partial u_j}\,=\,-\frac{\partial \vartheta_{i}^{\triangle_1}}{\partial u_j}-\frac{\partial \vartheta_{i}^{\triangle_2}}{\partial u_j}
\,=\,-\frac{\partial \vartheta_{j}^{\triangle_1}}{\partial u_i}-\frac{\partial \vartheta_{j}^{\triangle_2}}{\partial u_i}
\,=\,\frac{\partial K_j}{\partial u_i}\,<\,0.
\end{equation*}
In case that $e$ is a boundary edge, a similar argument implies
\[
\frac{\partial K_i}{\partial u_j}\,=\,\frac{\partial K_j}{\partial u_i}\,<\,0.
\]
This yields that the Jacobian matrix of $Th$ is symmetric.

Moreover, applying Lemma \ref{L-1-3}, we have
\[
\frac{\partial K_i}{\partial u_i}\,=\,-\sum\nolimits_{\triangle\in F(\{v_i\})}\frac{\partial \vartheta_{i}^{\triangle}}{\partial u_i}\,>\,0.
\]
Combining these, one obtains
\begin{equation}\label{E-11}
\Big|\frac{\partial K_i}{\partial u_i}\Big|-\sum\nolimits_{j\neq i}\Big|\frac{\partial K_i}{\partial u_j}\Big|\,=\,\sum\nolimits_{j=1}^{|V|}\frac{\partial K_i}{\partial u_j}\,=\,\sum\nolimits_{j=1}^{|V|}\frac{\partial K_j}{\partial u_i}\,=\,\sum\nolimits_{\triangle\in F(\{v_i\})}\frac{\partial\mathrm{Area}(\triangle)}{\partial u_i}\,>\,0.
\end{equation}
The Jacobian matrix is diagonally dominant and hence is positive definite.
\end{proof}

Let us consider the $1$-form $\omega=\sum\nolimits_{i=1}^{|V|}K_i du_i$. Because $\partial K_i/\partial u_j=\partial K_j/\partial u_i$, it is easy to see that $\omega$ is closed. Following Colin de Verdi\`{e}re \cite{Colin-invent}, the following energy function
\begin{equation}\label{E-12}
\Phi(u)\,=\,\int_{u(0)}^u \omega,
\end{equation}
is well-defined and is independent on the choice of
piecewise smooth paths in $\mathbb R^{|V|}_{-}$ from $u(0)$ to $u$. Here $u(0)\in\mathbb R^{|V|}_{-}$ is an arbitrary initial point.

\begin{lemma}\label{L-2-2}
The energy function $\Phi$ in terms of $u$ is strictly convex in $\mathbb{R}^{|V|}_{-}$.
\end{lemma}
\begin{proof}
Note that the Hessian of $\Phi$ is equal to the Jacobian of the curvature map, which is positive definite by Lemma \ref{L-2-1}. Hence $\Phi$ is strictly convex.
\end{proof}

\begin{corollary}\label{C-2-3}
The  curvature map $Th$ in terms of $u$ is injective.
\end{corollary}
\begin{proof}
Note that $\nabla \Phi =Th$ and $\Phi$ is strictly convex. The statement follows from the following Lemma \ref{L-2-4} from analysis.
\end{proof}

\begin{lemma}\label{L-2-4}
Suppose that $\Omega\subset \mathbb{R}^n$ is convex and the smooth function $h:\Omega\rightarrow \mathbb{R}$
is strictly convex. Then the gradient map $\nabla h:\Omega\rightarrow \mathbb{R}^n$ is injective.
\end{lemma}

\subsection{Image of the curvature map}

It is ready to characterize the image of the curvature map via the continuity method.

\begin{proof}[\textbf{Proof of Theorem \ref{T-0-4}}]
Let $Z\subset \mathbb R^{|V|}$ be the convex set characterized by the systems of inequalities (\ref{E-1}) and (\ref{E-2}). Here we assume that the curvature map $Th$ is in terms of $r$. It suffices to show that $Th(\mathbb R^{|V|}_{+})=Z$. We have the following claims:
\begin{itemize}
\item[$(i)$] $Th$ is continuous.
\item[$(ii)$] $Th$ is injective. Because the change of variables $r\mapsto u(r)$ is one-to-one, this follows from Corollary \ref{C-2-3}.
\item[$(iii)$] $Th(\mathbb R^{|V|}_{+})\subset Z$. Given any $r\in \mathbb R^{|V|}_{+}$, we need to show that $(K_1(r),\cdots,K_{|V|}(r))$ satisfies the two systems of inequalities. The verification of (\ref{E-1}) is trivial. To check (\ref{E-2}), assume that $A\subset V$ is a non-empty set. Let $E(A), F(A)$ be the sets of edges and triangles of $\mathcal T$ that have at least one vertex in $A$. For $j=1,2,3$, let $F_j(A)$  denote the set of triangles having $j$ vertices in $A$. Note that
\[
|F_1(A)|\,=\,|Lk(A)|.
\]
Using Lemma \ref{L-1-4}, one obtains
\[
\begin{aligned}
\sum\nolimits_{v_i\in A}\sigma(v_i)(r)\,<\,&\sum\nolimits_{(e,v)\in Lk(A)}\big(\pi-\Theta(e)\big)+\pi\big(\,|F_2(A)|+|F_3(A)|\,\big)\\
\,=\,&\sum\nolimits_{(e,v)\in Lk(A)}\big(\pi-\Theta(e)\big)+\pi|F(A)|-\pi|Lk(A)|.
\end{aligned}
\]
It follows that
\[
\sum\nolimits_{v_i\in A}K_i(r)
\,>\,-\sum\nolimits_{(e,v)\in Lk(A)}\big(\pi-\Theta(e)\big)+2\pi|A|-\pi|F(A)|+\pi|Lk(A)|-\pi| A\cap V_\partial|.
\]
Combining it with following Lemma \ref{L-2-5}, we deduce
\[
\sum\nolimits_{v_i\in A}K_i(r)
\,>\,-\sum\nolimits_{(e,v)\in Lk(A)}\big(\pi-\Theta(e)\big)+2\pi\chi(G(A)\setminus \partial S)+\pi\chi( G(A)\cap\partial S).
\]
\item[$(iv)$] The map $Th: \mathbb R^{|V|}_{+}\to Z$ is proper. It suffices to check that one of the inequalities in (\ref{E-1}) and (\ref{E-2}) becomes an equality as some of the radii tend to infinity or zero. To this end, let $\{r^n\}\subset \mathbb{R}_{+}^{|V|}$ be a sequence. In case that $r^n_i\to +\infty$ for some fixed $i\in\{1,2,\cdots,|V|\}$, we need to check that
    \[
    \sigma(v_i)(r^n)\,\to\,0,
    \]
 which is an immediate consequence of  Lemma \ref{L-1-6}. In case that $r^n_i\to 0$ for some  $i\in\{1,2,\cdots,|V|\}$, by Lemma \ref{L-2-5}, it suffices to show
 \[
 \sum\nolimits_{v_i\in A}K_i(r^n)
\,\to \,-\sum\nolimits_{(e,v)\in Lk(A)}\big(\pi-\Theta(e)\big)+2\pi|A|-\pi|F(A)|+\pi|Lk(A)|-\pi|A\cap V_\partial|,
 \]
where $A\subset V$ consists of all vertices $v_i\in V$ such that $r^n_i\to 0$. Equivalently, we need to check that
 \[
 \sum\nolimits_{v_i\in A}\sigma(v_i)(r^n)\,\to \,\sum\nolimits_{(e,v)\in Lk(A)}\big(\pi-\Theta(e)\big)+\pi\big(\,|F_2(A)|+|F_3(A)|\,\big)
 \]
Using Lemma \ref{L-1-5}, the claim follows.
\end{itemize}

By Brouwer's theorem on invariance of domain, the first three claims imply that $Th(\mathbb{R}_{+}^{|V|})$ is a non-empty open set of $Z$. Combining with the fourth claim, $Th(\mathbb{R}_{+}^{|V|})$ is both open and closed in $Z$. Because $Z$ is connected, we have $Th(\mathbb{R}_{+}^{|V|})=Z$.
\end{proof}

\begin{lemma} \label{L-2-5}
For any non-empty subset $A$ of $V$, we have
\[
2|A|-|F(A)|+|Lk(A)|-|A\cap V_\partial|\,=\,2\chi(G(A)\setminus \partial S)+\chi(G(A)\cap\partial S).
\]
\end{lemma}

We postpone the proof to Appendix.

\subsection{Completeness}
Let us consider the combinatorial Ricci flow (\ref{E-5}). Using the substitution $u_i=\ln \tanh (r_i/2)$ for $i=1,2,\cdots,|V|$, one
rewrites it as an autonomous ODE system
\begin{equation}\label{E-13}
\frac{du_i}{dt}\,=\,-K_i.
\end{equation}
For simplicity, we shall not distinguish the $r$-flow (\ref{E-5}) and the $u$-flow (\ref{E-13}), since there is no substantial difference between them.

Because each $K_i$ depends on $u$ smoothly, $(K_1,K_2,\cdots,K_{|V|})$ is locally Lipschitz continuous in $\mathbb{R}^{|V|}_{-}$.  By the well-known Picard theorem in classical ODE theory, the $u$-flow (\ref{E-13}) has a unique solution $u(t)$, where $t\in[0, \epsilon)$ for some $\epsilon>0$. Hence for any initial value $u(0)\in \mathbb{R}^{|V|}_{-}$, the solution to the $u$-flow (\ref{E-13}) uniquely exists in a maximal time interval $[0,T_0)$ with $0<T_0\leq +\infty$. As a result, the solution to the $r$-flow (\ref{E-5}) also uniquely exists in a maximal interval $[0,T_0)$. The next lemma shows that $T_0=+\infty$.

\begin{lemma}\label{L-2-6}
The flow (\ref{E-5}) has a unique solution $r(t)$ which exists for all time $t\geq0$. Moreover, each $r_i(t)$ is  bounded from above in $[0,+\infty)$.
\end{lemma}

\begin{proof}
Observe that each $|K_i|$ is uniformly bounded by a positive constant $c$, which depends only on the combinatorial information of the triangulation $\mathcal{T}.$ Using (\ref{E-13}), for each vertex $v_i\in V$, we derive
\[
-ct\,\leq\, u_i(t)-u_i(0)\,\leq\, ct.
\]
 It follows that
\[
r_i(t)\,\geq\,\ln \frac{1+m_ie^{-ct}}{1-m_ie^{-ct}}\,>\,0,
\]
where $m_i=\tanh (r_i(0)/2)\in(0,1)$.
This means  $r_i(t)$ never touches the $0$-boundary of $\mathbb{R}^{|V|}_{+},$ i.e. the topological boundary $\partial \mathbb{R}^{|V|}_{+}$, in any finite time interval.

It remains to show $r_i(t)$ is  bounded from above in $[0, T_0)$ for $i=1,2,\cdots,|V|$. Assume that it is not true. Then there exists at least one vertex $v_i\in V$ such that $\overline{\lim\nolimits}_{t\uparrow T_0}~r_i(t)=+\infty$.
For such $i$, by Lemma \ref{L-1-6}, there exists $L>0$ sufficiently large such that $K_i>\pi/2$ whenever $r_i>L$. Since $\overline{\lim\nolimits}_{t\uparrow T_0}~r_i(t)=+\infty$, we can choose $b\in(0,T_0)$ such that $r_i(b)>L$. Set
\[
a\,=\,\inf_{s}\,\{\,s\in [0,b)\,|\,r_i(t)>L,\,\forall t\in [s,b]\,\}.
\]
One is ready to see that $r_i(a)=L$. Moreover, for $t\in [a,b]$, note that
\[
r'_i(t)\,=\,-K_i\sinh r_i\,<\,0.
 \]
Hence $r_i(b)\leq r_i(a)=L$, which contradicts to $r_i(b)>L$. As a result, each $r_i(t)$ has an  upper bound in $[0,T_0)$. It follows that $r(t)$ never touches the $+\infty$-boundary of $\mathbb{R}^{|V|}_{+}$.

To summarise, $r(t)$ stays in a compact set in $\R^{|V|}_+$. The classical ODE theory then implies $T_0=+\infty$. Consequently, the flow (\ref{E-5}) exists for all time $t\geq 0$ and each $r_i(t)$ is bounded from above in $[0,+\infty)$.
\end{proof}

\subsection{Convergence}
\begin{proposition}\label{P-2-7}
 If the solution $r(t)$ to the flow (\ref{E-5}) converges to a vector $r^\ast\in \mathbb R^{|V|}_{+}$, then  $r^\ast$ produces a smooth hyperbolic metric on $S$.
\end{proposition}

\begin{proof}
Obviously, $u(t)$ converges to $u^\ast$, where $u^\ast\in\mathbb R^{|V|}_{-}$ is the vector corresponding to $r^\ast$.
For any positive integer $n$, by the mean value theorem there exists $\zeta_n\in(n,n+1)$ such that
\[
u_i(n+1)-u_i(n)\,=\,u'_i(\zeta_n)\,=\,-K_i(u(\zeta_n)).
\]
As $n\rightarrow+\infty$, we have $u_i(n+1)-u_i(n)\to 0$ and $K_i(u(\zeta_n))\to K_i(u^\ast)$. Thus
 \[
 K_i(u^\ast)\,=\,0,
 \]
 which concludes the proposition.
\end{proof}

The following lemma is a standard result from the analysis. See, e.g., \cite{Ge-thesis,Ge-Hua}.
\begin{lemma}\label{L-2-8}
Let $h$ be a strictly convex smooth function defined in a convex set $\Omega\subset\R^n$ with a critical point $p\in \Omega$. Then the following two properties hold:
\begin{itemize}
\item[$(i)$] $p$ is the unique global minimum point of $h$.
\item[$(ii)$] If $\Omega$ is unbounded, then $
\lim\limits_{\|x\|\to+\infty}h(x)\,=\,+\infty$.
\end{itemize}
\end{lemma}

\begin{proof}[\textbf{Proof of Theorem \ref{T-0-7}}]
By Proposition \ref{P-2-7}, the implication $``H_1\Rightarrow H_2"$ is straightforward.

To see $``H_2\Rightarrow H_1"$, let $r^\ast$ be the radius vector such that $K_i(r^\ast)=0$ for $i=1,2,\cdots,|V|$ and let $u^\ast$ be the vector corresponding to $r^\ast$. Consider the energy function $\Phi$ defined as (\ref{E-12}). Then $\nabla \Phi|_{u^\ast}=0$. That means $u^\ast$ is a critical point of $\Phi$. Because $\Phi$ is strictly convex, by Lemma \ref{L-2-8}, $u^\ast$ is a global minimal point of $\Phi$. Meanwhile,
\[
\frac{d\Phi(u(t))}{dt}\,=\,-\sum\nolimits_{i=1}^{|V|}\frac{\partial\Phi}{\partial u_i} u_i'(t)\,=\,-\sum\nolimits_{i=1}^{|V|}K_i^2\,\leq\,0.
\]
It follows that $\Phi(u(t))$ is descending and bounded from below. Hence $\Phi(u(+\infty))$ exists.

We claim that $\{u(t):t\in[0,+\infty)\}$ is compact in $\mathbb{R}^{|V|}_-$. On the one hand, each $u_i(t)$ is bounded from below. Otherwise, there exists $t_n\to+\infty$ such that $\|u(t_n)\|\to +\infty$. Due to Lemma \ref{L-2-8},
\[
\lim_{n\to\infty}\Phi(u(t_n))\,=\,+\infty,
\]
which contradicts to the convergence of $\Phi(u(t))$. On the other hand, for each $u_{i}(t)$, by Lemma \ref{L-2-6}, there exists a constant $L>0$ such that
\[
u_{i}(t)\,=\,\ln\tanh \frac{r_{i}(t)}{2}\,<\,\ln\tanh \frac{L}{2}.
\]
Because $u(t)$ is bounded from below and above, we prove the claim.

It remains to show that $u(t)$ converges as $t\to+\infty$. For each positive integer $n$, one can choose $\zeta_n\in(n,n+1)$ satisfying
\[
\Phi(u(n+1))-\Phi(u(n))\,=\,\Phi'(u(\zeta_n))\,=\,-\sum\nolimits_{i=1}^{|V|}K_i^2(u(\zeta_n).
\]
As $n\to\infty$, $\Phi(u(n+1))-\Phi(u(n))\to 0$, we obtain
\[
K_i(u(\zeta_n))\,\to\, 0
\]
for $i=1,2,\cdots,|V|$. Let us pick up a subsequence $\{u(\zeta_{n_k})\}$ convergent to $\tilde{u}\in\mathbb{R}^{|V|}_{-}$. It is easy to see that $\tilde{u}$ is also a critical point. Using Lemma \ref{L-2-8}, we have  $\tilde{u}=u^\ast$, which implies $u(\zeta_{n_k})\to u^\ast$. Moreover, a similar argument shows that any convergent subsequence of $u(t)$ tends to $u^\ast$. As a result, $u(t)\to u^\ast$.

In summary, we show that $``H_1\Leftrightarrow H_2"$. Meanwhile, $``H_2\Leftrightarrow H_3"$ is an immediate consequence of Theorem \ref{T-0-4}. Finally, the relation $``H_3\Leftrightarrow H_4"$ follows from  Gauss-Bonnet inequality (\ref{E-4}) and the following Proposition \ref{P-2-9}.

It remains to consider the exponential convergence of the flow. Using (\ref{E-13}), we have
\[
\frac{dK_i}{d t}\,=\,-\sum\nolimits_{j=1}^{|V|}\frac{\partial K_i}{\partial u_j}K_j
\]
Set
\[
M(t)\,=\,\max\big\{\,K_1(r(t)),\cdots,K_{|V|}(r(t))\,\big\}.
\]
Applying (\ref{E-11}), a routine computation gives
\[
\frac{d M(t)}{dt}\,\leq\,-\;M(t) \sum\nolimits_{\triangle\in F(\{v_l\})}\frac{\partial\mathrm{Area}(\triangle)}{\partial u_l},
\]
where $l\in \{1,2,\cdots,|V|\}$ satisfies $K_l(r(t))=M(t)$. Since $\{u(t):t\in[0,+\infty)\}$ is compact, standard arguments imply that there exists a positive number $c_1$ such that
\[
M(t)\leq\,M(0)e^{-c_1t}.
\]
Hence there exists $c_2>0$ satisfying
\[
\frac{d r_i}{d t}\,=\,-K_i\sinh r_i\,\geq\,-c_2e^{-c_1t},
\]
It follows that
\[
r_i(t)-r_i^\ast\,=\,-\int\nolimits_{t}^{+\infty}r'_i(s)\,ds\,\leq \frac{c_2}{c_1}e^{-c_1t}.
\]
Similarly, there exist positive numbers $c_3,c_4$ such that
\[
r_i(t)-r_i^\ast\,\geq \frac{c_4}{c_3}e^{-c_3t}.
\]
We thus finish the proof.
\end{proof}

The following combinatorial fact is a special case of Proposition \ref{P-4-3} in Section \ref{Further}.
\begin{proposition}\label{P-2-9}
$H_3$  holds for all proper non-empty subset $A$ of $V$ if and only if the conditions $\mathbf{(C2)}$, $\mathbf{(C3)}$ are satisfied.
\end{proposition}

\section{Euclidean background geometry}\label{sec:euclidean}

\subsection{Curvature map and energy function}
For $i=1,2,\cdots,|V|$, using the change of variables $u_i=\ln r_i$, one regards the vertex curvatures as smooth functions of $u=(u_1,\cdots,u_{|V|})$.

\begin{lemma}\label{L-3-1}
The Jacobian matrix of  the curvature map in terms of $u$ is symmetric and positive semi-definite. Moreover, when restricted in $\Xi_d=\{u\in\mathbb R^{|V|}:\sum_{i=1}^{|V|} u_i= d\}$ ($d$ is a given real number), it is  positive definite.
\end{lemma}
\begin{proof}
By Lemma \ref{L-1-3}, we have
\begin{equation*}
\frac{\partial K_i}{\partial u_i} \,>\, 0,\;\;\;\;\frac{\partial K_i}{\partial u_j}\,=\,\frac{\partial K_j}{\partial u_i} \,\leq\, 0\, (j\neq i),\;\;\;\; \sum\nolimits_{j=1}^{|V|}\frac{\partial K_i}{\partial u_j}\,=\,0.
\end{equation*}
Similar arguments to the proof of Lemma \ref{L-2-1} yield the first part. And it follows from the linear algebra that the Jacobian matrix of the curvature map is positive definite when restricted in $\Xi_d$.
\end{proof}

Let us consider the $1$-form $\omega=\sum\nolimits_{i=1}^{|V|}(K_i-K_{av}) du_i$. Similarly, $\omega$ is closed and the following energy function
\begin{equation}\label{E-14}
\Phi(u)\,=\,\int_{u(0)}^u \omega,
\end{equation}
is well-defined and is independent on the choice of
piecewise smooth paths in $\mathbb R^{|V|}$ from an initial point $u(0)$ to $u$.

\begin{lemma}\label{L-3-2}
The energy function $\Phi$ in terms of $u$ is convex in $\mathbb{R}^{|V|}$ and is strictly convex when restricted in $\Xi_d$.
\end{lemma}
\begin{proof}
Because the Hessian of $\Phi$ is equal to the Jacobian  of the curvature map, the statement follows from Lemma \ref{L-3-1}.
\end{proof}

\begin{corollary}\label{C-3-3}
The curvature map in terms of $u$ is injective when restricted in $\Xi_d$.
\end{corollary}
\begin{proof}
It is an immediate consequence of Lemma \ref{L-2-4} and Lemma \ref{L-3-2}.
\end{proof}

\subsection{Image of the curvature map}
\begin{proof}[\textbf{Proof of Theorem \ref{T-0-5}}]
Let $Y\subset \mathbb R^{|V|}$ be the convex set characterized by the systems of inequalities (\ref{E-1}) and (\ref{E-3}). Set
\[
\Lambda_d\,=\,\Big\{\,r\in\mathbb R^{|V|}_{+}:\prod\nolimits_{i=1}^{|V|} r_i= e^d\,\Big\}.
\]
Apparently, $Th(\mathbb R^{|V|}_{+})=Th(\Lambda_d)$. Let us consider the restriction map
\[
Th:\,\Lambda_d\,\to\, Y.
\]
It is easy to see
\[
\dim (\Lambda_d)\,=\,\dim (Y)\,=\,|V|-1.
\]
Moreover, we have the following claims:
\begin{itemize}
\item[$(i)$] $Th$ is continuous.
\item[$(ii)$] $Th$ is injective, which is a result of Corollary \ref{C-3-3}.
\item[$(iii)$] $Th(\Lambda_d)\subset Y$. For any $r\in \Lambda_d$, Lemma \ref{L-1-4} implies that
\[
\begin{aligned}
\sum\nolimits_{v_i\in A}\sigma(v_i)(r)\,\leq&\,\sum\nolimits_{(e,v)\in Lk(A)}\big(\pi-\Theta(e)\big)+\pi\big(\,|F_2(A)|+|F_3(A)|\,\big)\\
\,=&\,\sum\nolimits_{(e,v)\in Lk(A)}\big(\pi-\Theta(e)\big)+\pi|F(A)|-\pi|Lk(A)|.
\end{aligned}
\]
Hence
\[
\sum\nolimits_{v_i\in A}K_i(r)
\,\geq\,-\sum\nolimits_{(e,v)\in Lk(A)}\big(\pi-\Theta(e)\big)+2\pi|A|-\pi|F(A)|+\pi|Lk(A)|-\pi| A\cap V_\partial|,
\]
where the equality holds if and only if $F(A)=F_3(A)$. Noting that $F(A)=F_3(A)$ is equivalent to $V=A$, the statement follows from Lemma \ref{L-2-5}.

\item[$(iv)$] The map $Th: \Lambda_d\to Y$ is proper. Because of Remark \ref{R-3-4}, it remains to check that (\ref{E-3}) becomes an equality when the radii for those vertices in $A$ tend to zero, where $A$ is an arbitrary proper non-empty subset of $V$. Due to Lemma \ref{L-1-5}, the claim holds.
\end{itemize}

By continuity method, similar arguments to the proof of Theorem \ref{T-0-4} imply that $Th(\Lambda_d)=Y$. Thus the theorem is proved.
\end{proof}

\begin{remark}\label{R-3-4}
For each $v_i\in V_\partial$, set $A_i=V\setminus \{v_i\}$. It follows from (\ref{E-3})  that
\[
\sum\nolimits_{j\neq i}K_j\,>\,2\pi\chi(S\setminus \partial S)+\pi\chi(\partial S\setminus \{v_i\})\,=\,2\pi\chi(S)-\pi.
\]
Combining with the Gauss-Bonnet equality in (\ref{E-4}), we obtain $K_i<\pi$. Similarly, $K_i<2\pi$ if $v_i\in V\setminus V_\partial$. This means that (\ref{E-1}) is implied by (\ref{E-3}).
\end{remark}

\subsection{Completeness}
Let us consider the flow (\ref{E-6}). Similarly, one rewrites it as an autonomous ODE system
\begin{equation}\label{E-15}
\frac{du_i}{dt}\,=\,K_{av}-K_i.
\end{equation}
And we shall not distinguish the $r$-flow (\ref{E-6}) and the $u$-flow (\ref{E-15}).

Picard's theorem implies that the flow (\ref{E-15}) has a unique solution $u(t)$, where $t\in[0, \epsilon)$ for some $\epsilon>0$. Therefore, for any initial value $u(0)\in \mathbb{R}^{|V|}$, the solution to the $u$-flow (\ref{E-15}) uniquely exists in a maximal time interval $[0,T_0)$ with $0<T_0\leq +\infty$. And the solution to the $r$-flow (\ref{E-6}) also uniquely exists in a maximal interval $[0,T_0)$. Similarly, we have $T_0=+\infty$.

\begin{lemma}\label{L-3-5}
For any $r(0)\in\mathbb{R}^{|V|}_{+}$, the combinatorial Ricci flow (\ref{E-6}) has a unique solution $r(t)$ which exists for all time $t\geq0$. Moreover, $\prod\nolimits_{i=1}^{|V|} r_i(t)$ is invariant as $t$ varies.
\end{lemma}

\begin{proof}
First, a simple calculation gives
\[
\sum\nolimits_{i=1}^{|V|}\frac{du_i}{dt}\,=\,\sum\nolimits_{i=1}^{|V|}(K_{av}-K_i)\,=\,0.
\]
Hence $\sum\nolimits_{i=1}^{|V|} u_i(t)$ stays constant, which implies that $\prod\nolimits_{i=1}^{|V|} r_i(t)$ is invariant as $t$ varies. To prove that the solution exists for all time $t\geq 0$, one needs to show that $r(t)$ never touches the boundary. Because
\[
\prod\nolimits_{i=1}^{|V|} r_i(t)\,=\,\prod\nolimits_{i=1}^{|V|} r_i(0),
\]
it suffices to check that every $r_i(t)$ can not become zero within any finite time. Observe that each $|K_{av}-K_i|$ is uniformly bounded  by a positive constant $c$, which depends only on the combinatorial information of the triangulation $\mathcal{T}.$ Using (\ref{E-15}),  for each $i$, we have
\[
-ct\,\leq\, u_i(t)-u_i(0)\,\leq\, ct.
\]
 It follows that
\[
r_i(t)\,\geq\,r_i(0)e^{-ct}\,>\,0.
\]
Thus the lemma is proved.
\end{proof}

\subsection{Convergence}

\begin{proposition}\label{P-3-6}
 If the solution to the flow (\ref{E-6}) converges to a vector $r^\ast\in \mathbb R^{|V|}_{+}$, then $r^\ast$ produces an Euclidean cone metric on $S$ so that all vertex curvatures are equal to $K_{av}$.
\end{proposition}
The proof follows verbatim from Proposition \ref{P-2-7}. We omit the details.

\begin{proof}[\textbf{Proof of Theorem \ref{T-0-8}}]The implication $``E_1\Rightarrow E_2"$ follows from Proposition \ref{P-3-6}.

To show $``E_2\Rightarrow E_1"$, suppose that $r^\ast\in\Lambda(0)$ satisfies $K_i(r^\ast)=K_{av}$ for $i=1,\cdots,|V|$. Let $u^\ast$ correspond to $r^\ast$. Then $u^\ast\in \Xi(0)$, where
\[
\Xi(0)\,=\,\big\{u\in\mathbb R^{|V|}:\,\sum\nolimits_{i=1}^{|V|}u_i=\sum\nolimits_{i=1}^{|V|}\ln r_i(0)\big\}.
\]
Consider the energy function $\Phi$ defined as (\ref{E-14}). Then $\nabla \Phi|_{u^\ast}=0$, which implies that $u^\ast$ is a critical point of $\Phi$. By Lemma \ref{L-3-2} and Lemma \ref{L-2-8}, $u^\ast$ is a global minimal point of $\Phi$ in $\Xi(0)$. It follows  that $\Phi$ is bounded from below. Moreover,
\[
\frac{d\Phi(u(t))}{dt}\,=\,-\sum\nolimits_{i=1}^{|V|}\frac{\partial\Phi}{\partial u_i} u_i'(t)\,=\,-\sum\nolimits_{i=1}^{|V|}(K_i-K_{av})^2\,\leq\,0.
\]
As a result, $\Phi(u(+\infty))$ exists. We claim that $\{u(t):t\in[0,+\infty)\}$ is compact in $\Xi(0)$. It suffices to prove that $u(t)$ is bounded. Suppose that it is not true. Then there exists $t_n\to+\infty$ such that $\|u(t_n)\|\to +\infty$. By Lemma \ref{L-2-8}, it is not hard to see
\[
\lim_{n\to\infty}\Phi(u(t_n))\,=\,+\infty,
\]
which contradicts to the convergence of $\Phi(u(t))$.
Hence $\{u(t):t\in[0,+\infty)\}$ is compact. And similar arguments to the proof of Theorem \ref{T-0-7} imply that $u(t)\to u^\ast$.

To summarise, we show that $``E_1\Leftrightarrow E_2"$. Meanwhile, the relation $``E_2\Leftrightarrow E_3"$ follows from Theorem \ref{T-0-5}. It remains to investigate the exponential convergence of the flow. By Gauss-Bonnet equality, one obtains
\[
\sum\nolimits_{j=1}^{|V|}\frac{\partial K_i}{\partial u_j}\,=\,\sum\nolimits_{j=1}^{|V|}\frac{\partial K_j}{\partial u_i}\,=\,\frac{\partial}{\partial u_i}\big(2\pi\chi(S)\big)\,=\,0.
\]
It follows that
\begin{equation*}
\begin{aligned}
\frac{dK_i}{d t}\,=&\,-\sum\nolimits_{j=1}^{|V|}\frac{\partial K_i}{\partial u_j}(K_j-K_{av})\\
\,=&\,-\sum\nolimits_{j=1}^{|V|}\frac{\partial K_i}{\partial u_j}(K_j-K_{av})+\sum\nolimits_{j=1}^{|V|}\frac{\partial K_i}{\partial u_j}(K_i-K_{av})\\
\,=&\,\sum\nolimits_{j=1}^{|V|}\frac{\partial K_i}{\partial u_j}(K_i-K_j).
\end{aligned}
\end{equation*}
Set $g(t)=\sum\nolimits_{i=1}^{|V|}(K_i-K_{av})^2$. A routine calculation gives
\[
\begin{aligned}
\frac{dg(t)}{dt}\,=&\,\sum\nolimits_{i=1}^{|V|}\sum\nolimits_{j=1}^{|V|}\frac{\partial K_i}{\partial u_j}(K_{i}-K_{av})(K_i-K_j)\\
\,=&\,\sum\nolimits_{i=1}^{|V|}\sum\nolimits_{j=1}^{|V|}\frac{\partial K_j}{\partial u_i}(K_{j}-K_{av})(K_j-K_i)\\
\,=&\,\frac{1}{2}\sum\nolimits_{i=1}^{|V|}\sum\nolimits_{j=1}^{|V|}\frac{\partial K_i}{\partial u_j}(K_i-K_j)^2
\end{aligned}
\]
Recall that the coefficient $\partial K_i/\partial u_j=0$ if the vertices $v_i, v_j$ are not adjacent. Therefore,
\[
\frac{dg(t)}{dt}\,=\,\sum\nolimits_{[v_i,v_j]\in E}\frac{\partial K_i}{\partial u_j}(K_i-K_j)^2.
\]
By the following Lemma \ref{L-3-7}, we show that
\[
\frac{dg(t)}{dt}\,\leq\,-c_1g(t).
\]
As a result,
\[
g(t)\,\leq\, g(0)e^{-c_1t}.
\]
It follows that $|K_i(t)-K_{av}|^2\leq g(t)\leq g(0)e^{-c_1t}$. By taking the integration in (\ref{E-5}), we prove that $r(t)$ converges exponentially fast to $r^\ast$.
\end{proof}

\begin{lemma}\label{L-3-7}
There exists $c_1>0$ such that
\[
\sum\nolimits_{[v_i,v_j]\in E}\frac{\partial K_i}{\partial u_j}(K_i-K_j)^2\,\leq\,-c_1\sum\nolimits_{i=1}^{|V|}(K_i-K_{av})^2.
\]
\end{lemma}

In Section \ref{Further}, we will prove Lemma \ref{L-4-2}, which generalizes the above lemma.

\begin{proof}[\textbf{Proof of the Theorem \ref{T-0-3}}]
In case that $\chi(S)<0$, the statement is an immediate result of Theorem \ref{T-0-7}. Suppose that $\chi(S)=0$. Then $K_{av}=2\pi\chi(S)/|V|=0$. Applying Theorem \ref{T-0-8}, the vector $(0,0,\cdots,0)$ belongs to the image of the curvature map if and only if $E_3$ holds. Because $\chi(S)=0$, $E_3$ and $H_3$ are the same for any proper non-empty subset $A$ of $V$. Due to  Proposition \ref{P-2-9}, we finish the proof.
\end{proof}

\section{Further discussions}\label{Further}

\subsection{Circle patterns with prescribed curvatures}

To search the circle pattern with other prescribed vertex curvatures, we modify (\ref{E-5}) to  another ODE system
\begin{equation}\label{E-16}
\frac{dr_i}{dt}\,=\,(k_i-K_i)\sinh r_i
\end{equation}
and modify (\ref{E-6}) to the ODE system
\begin{equation}\label{E-17}
\frac{dr_i}{dt}\,=\,(k_i-K_i)r_i
\end{equation}
with an initial radius vector $r(0)\in \mathbb R^{|V|}_{+}$, where
$K=(k_1,\cdots,k_{|V|})\in \mathbb R^{|V|}$ is a prescribed curvature vector. Such a vector $K$ is called hyperbolic (resp. Euclidean)
if it satisfies  (\ref{E-1}) and the Gauss-Bonnet inequality (resp. equality) in (\ref{E-4}). And $K$ is called \textbf{attainable}, if it satisfies (\ref{E-2}) except for the case that $A=V$. Remind that (\ref{E-2}) and (\ref{E-3}) are the same if $A$ is a proper non-empty subset of $V$.

\begin{theorem}\label{T-4-1}
Suppose that $\Theta:E\to [0,\pi)$ satisfies $\mathbf{(C1)}$. For any hyperbolic (resp. Euclidean) curvature vector $K=(k_1,\cdots,k_{|V|})$, the flow (\ref{E-16}) (resp. (\ref{E-17})) has a unique solution $r(t)$ which exists for all time $t\geq 0$. And $r(t)$ converges if and only if $K$ is \textbf{attainable}. Moreover, if $r(t)$ converges, then it converges exponentially fast to a radius vector which produces a hyperbolic (resp. Euclidean) cone metric $\mu$ on $S$, so that $(S,\mu)$ has vertex curvatures assigned by $K$ and supports a $\mathcal T$-type circle pattern $\mathcal P$ whose exterior intersection angle function is $\Theta$.
\end{theorem}

\begin{proof}
{The completeness follows verbatim from  Lemma \ref{L-2-6} and Lemma \ref{L-3-5}. Moreover, using similar arguments to proofs of the relations $``H_1\Leftrightarrow H_2\Leftrightarrow H_3"$ and $``E_1\Leftrightarrow  E_2\Leftrightarrow E_3"$, we prove that $r(t)$ converges if and only if $K$ is \textbf{attainable}}.

It remains to deal with the exponentially convergence part.
In hyperbolic background geometry, using (\ref{E-16}), we have
\[
\frac{dK_i}{d t}\,=\,-\sum\nolimits_{j=1}^{|V|}\frac{\partial K_i}{\partial u_j}(K_j-k_j).
\]
Set
\[
M(t)\,=\,\max\big\{\,K_1(r(t))-k_1,\cdots,K_{|V|}(r(t))-k_{|V|}\,\big\}.
\]
By applying (\ref{E-11}), a routine computation gives
\[
\frac{d M(t)}{dt}\,\leq\,-\;M(t) \sum\nolimits_{\triangle\in F(\{v_l\})}\frac{\partial\mathrm{Area}(\triangle)}{\partial u_l},
\]
where $l\in \{1,2,\cdots,|V|\}$ satisfies $K_l(r(t))-k_l=M(t)$. By combining it with the last step of the proof of Theorem \ref{T-0-7}, the statement follows.

In Euclidean background geometry, let $g(t)=\sum\nolimits_{i=1}^{|V|}(K_i-k_i)^2$. It is easy to see
\[
\frac{dg(t)}{dt}\,=\,\frac{1}{2}\sum\nolimits_{i=1}^{|V|}\sum\nolimits_{j=1}^{|V|}\frac{\partial K_i}{\partial u_j}(K_i-k_i-K_j+k_j)^2.
\]
Note that the coefficient $\partial K_i/\partial u_j=0$ if the vertices $v_i, v_j$ are not adjacent. Therefore,
\[
\frac{dg(t)}{dt}\,=\,\sum\nolimits_{[v_i,v_j]\in E}\frac{\partial K_i}{\partial u_j}(K_i-k_i-K_j+k_j)^2.
\]
By the following Lemma \ref{L-4-2} and similar arguments to the proof of Theorem \ref{T-0-8}, we prove that $r(t)$ converges exponentially fast to the desired radius vector.
\end{proof}

\begin{lemma}\label{L-4-2}
Let $K=(k_1,\cdots,k_{|V|})$ be an Euclidean curvature vector. In Euclidean background geometry, there exists a positive number $c_1$ such that
\[
\sum\nolimits_{[v_i,v_j]\in E}\frac{\partial K_i}{\partial u_j}(K_i-k_i-K_j+k_j)^2\,\leq\,-c_1\sum\nolimits_{i=1}^{|V|}(K_i-k_i)^2.
\]
\end{lemma}

\begin{proof}
We follow the method of Chow-Luo \cite{ChowLuo-jdg}. Because $r(t)$ stays in a compact set, the coefficients $\partial K_i/\partial u_j$ in the above inequality are uniformly bounded from above by a negative constant. It suffices to prove
\[
\sum\nolimits_{i=1}^{|V|}(K_i-k_i)^2\,\leq\, c_2\sum\nolimits_{[v_i,v_j]\in E}(K_i-k_i-K_j+k_j)^2
\]
for some positive number $c_2$. By the Gauss-Bonnet equality, we have
\[
\sum\nolimits_{j=1}^{|V|}(K_j-k_j)\,=\,2\pi\chi(S)-2\pi\chi(S)\,=\,0.
\]
It follows that
\[
\begin{aligned}
(K_i-k_i)^2\,=&\,
\Big[K_i-k_i-1/|V|\sum\nolimits_{j=1}^{|V|}\big(K_j-k_j\big)\Big]^2\\
\,=&\,\Big[\sum\nolimits_{j=1}^{|V|}1/|V|\big(K_i-k_i-K_j+k_j\big)\Big]^2\\
\,\leq& \,1/|V|\sum\nolimits_{j=1}^{|V|}(K_i-k_i-K_j+k_j)^2.
\end{aligned}
\]
Hence
\[
\sum\nolimits_{i=1}^{|V|}(K_i-k_i)^2\,\leq\,1/|V|\sum\nolimits_{i=1}^{|V|}\sum\nolimits_{j=1}^{|V|}(K_i-k_i-K_j+k_j)^2.
\]
Meanwhile, there exists a sequence of vertices $v_{l_0}=v_i,
v_{l_1},\cdots,v_{l_q} = v_j$ such that $v_{l_h}$ and $v_{l_{h+1}}$ are adjacent, where $q$ depends
only on $v_i,v_j$ and the triangulation $\mathcal{T}.$ By Cauchy's inequality,
\[
\begin{aligned}
(K_i-k_i-K_j+k_j)^2\,=&\, \Big[\sum\nolimits_{h=0}^{q-1}\big(K_{l_h} -k_{l_h}- K_{l_{h+1}}+k_{l_{h+1}}\big)\Big]^2\\
\leq&\, q \sum\nolimits_{h=0}^{q-1}\big(K_{l_h} -k_{l_h}- K_{l_{h+1}}+k_{l_{h+1}}\big)^2\\
\leq&\,q\sum\nolimits_{[v_\alpha,v_\beta]\in E}\big(K_\alpha-k_\alpha-K_\beta+k_\beta\big)^2.
\end{aligned}
\]
Choose an upper bound $D$ for all such $q$. Combining the above relations, we have
\[
\sum\nolimits_{i=1}^{|V|}(K_i-k_i)^2\,\leq\, D|V|\sum\nolimits_{[v_i,v_j]\in E}(K_i-k_i-K_j+k_j)^2.
\]
Thus the lemma is proved.
\end{proof}

\begin{proof}[\textbf{Proof of Theorem \ref{T-0-9}}]
The theorem is an immediate consequence of Theorem \ref{T-4-1} and the following Proposition \ref{P-4-3}.
\end{proof}

Given a function $\varphi: V_\partial\to [0,\pi)$, let us define a vector $K[\varphi]=(k_1,\cdots,k_{|V|})$ by
\[
k_i\,=\,
\begin{cases}
\; \varphi(v_i), & \text{if}\; v_i\in V_\partial;\\
\; 0, & \text{if}\; v_i\in V\setminus V_\partial.
\end{cases}
\]
The following result plays a crucial role in relating Theorem \ref{T-0-9} to Theorem \ref{T-4-1}.
\begin{proposition}\label{P-4-3}
$K[\varphi]$ is \textbf{attainable} if and only if the conditions $\mathbf{(C2)}$, $\mathbf{(BV2)}$ are satisfied.
\end{proposition}
The proof is based on combinatorial arguments and  is postponed to Appendix.

\subsection{Planar circle patterns}

Setting $(g,n)=(0,1)$, Theorem \ref{T-0-8}  gives the following result.
\begin{theorem}\label{T-4-3}
 Let $\mathcal T$ be a triangulation of a topological closed disk. Suppose that $\Theta:E\to [0,\pi)$ and $\varphi: V_\partial\to [0,\pi)$ are two functions such that $\mathbf{(C1)}$, $\mathbf{(C2)}$ and the following conditions are satisfied:
 \begin{itemize}
 \item[$\mathbf{(Q1)}$] The Gauss-Bonnet inequality (resp. equality) holds:
\[
\sum\nolimits_{v\in V_\partial}\varphi(v) \,>\,2\pi \quad \big(\;\text{resp.}\;\sum\nolimits_{v\in V_\partial}\varphi(v)\,=\,2\pi\;\big).
\]
\item[$\mathbf{(Q2)}$] If the edges $e_1,\cdots,e_s$ form a \textbf{semi pseudo-Jordan} arc $\gamma$ joining two distinct boundary vertices, then for any non-empty \textbf{semi enclosing set} $W$ of $\gamma$
\[
\sum\nolimits_{v\in W\cap V_\partial}\varphi(v)+
\sum\nolimits_{l=1}^s\big(\pi-\Theta(e_l)\big)\,>\,\pi.
\]
\end{itemize}
Then there exists a convex hyperbolic (resp. Euclidean) polygon $Q$ which has external angles given by $\varphi$ and supports a $\mathcal T$-type circle pattern $\mathcal P$ with exterior intersection angles given by $\Theta$. Moreover, the pair $(Q,\mathcal P)$ is unique up to isometries (resp. similarities).
\end{theorem}

Turning to the limiting case that $\varphi\equiv \pi$, the following result is straightforward.
\begin{corollary}\label{C-4-4}
 Let $\mathcal T$ be a triangulation of a topological closed disk. Suppose that  $\Theta:E\to [0,\pi)$ is a function satisfying $\mathbf{(C1)}$, $\mathbf{(C2)}$. Then in the hyperbolic disk there exists a $\mathcal T$-type circle pattern $\mathcal P$  whose exterior intersection angle function is $\Theta$ and whose boundary circles are horocyles. Moreover, $\mathcal P$ is unique up to isometries.
\end{corollary}

Another corollary is a generalization of Marden-Rodin's theorem \cite{Marden-Rodin}.

\begin{theorem}\label{T-4-5}
Let $\mathcal T$ be a triangulation of the sphere. Suppose that $\Theta:E\to [0,\pi)$ satisfies $\mathbf{(C1)}$ and the following conditions:
\begin{itemize}
\item[$(i)$]There exists a triangle of $\mathcal T$ with boundary edges $e_a,e_b,e_c$ such that $\sum_{l=a,b,c}\Theta(e_l)<\pi$.
\item[$(ii)$] If the edges $e_1,e_2,\cdots,e_s$ form a simple closed curve which separates the vertices of $\mathcal T$, then $\sum_{l=1}^s\Theta(e_l)<(s-2)\pi$
\end{itemize}
Then there exists a $\mathcal T$-type circle pattern $\mathcal P$ on the Riemann sphere $\hat{\mathbb C}$ with the exterior intersection angles given by $\Theta$. Moreover, $\mathcal P$ is unique up to M\"{o}bius transformations.
\end{theorem}

A circle pattern $\mathcal P$ on the sphere can be stereographically projected to the plane. Since stereographic projection is conformal, the exterior intersection angles are the same. More precisely, one can construct the circle pattern on the sphere by constructing the corresponding planar pattern and then projecting it back to the sphere. Therefore, to prove Theorem \ref{T-4-5}, one needs only to construct a proper planar circle pattern.

\begin{proof}[\textbf{Proof of Theorem \ref{T-4-5}}]
Let $\bigtriangleup_\infty$ be the triangle of $\mathcal T$ whose boundary edges are $e_a,e_b,e_c$ and let $v_a,v_b,v_c$ be the vertices opposite to $e_a,e_b,e_c$. By stereographic projection, then $\mathcal T'=\mathcal T\setminus \bigtriangleup_\infty$ gives a triangulation of $\triangle_\infty$. Moreover, one can regard $\Theta:E\to [0,\pi)$ as a function defined on the edge set of $\mathcal T'$. For existence part,  due to the stereographic projection construction, it suffices to show that there exists a $\mathcal T'$-type circle pattern $\mathcal P$ on the unit disk  with the exterior intersection angles given by $\Theta$, which is a result of Corollary \ref{C-4-4}. For rigidity part, note that a M\"{o}bius transformation fixes the unit disk if and only if it is an isometry of hyperbolic disk. The statement follows from the rigidity part of Corollary \ref{C-4-4}.
\end{proof}

Finally, let us consider the circle patterns on rectangles.
\begin{proof}[\textbf{Proof of Corollary \ref{C-0-10}}]
Let $V_{*}\subset V_\partial$ be the set of corner vertices. We define a function $\varphi: V_\partial \to [0,\pi)$ by setting
\[
\varphi(v)\,=\,
\begin{cases}
\; \pi/2, & \text{if}\; v\in V_\ast;\\
\; 0, & \text{if}\; v\in V_\partial \setminus V_\ast.
\end{cases}
\]
Under the conditions \textbf{(R1)}, \textbf{(R2)}, it is easy to see that \textbf{(Q1)} and \textbf{(Q2)} are satisfied. By Theorem \ref{T-4-3}, the conclusion holds.
\end{proof}

%\appendix
\section{Appendix}\label{Appendix}

This section is devoted to several combinatorial facts. Specifically, we  prove  Lemma \ref{L-2-5} and Proposition \ref{P-4-3}.

\begin{proof}[\textbf{Proof of Lemma \ref{L-2-5}}]
Let $E_\partial$ be the set of boundary edges of $\mathcal T$. Set
\[
E_\partial (A)\,=\,E(A)\cap E_\partial,\quad E_O(A)\,=\,E(A)\setminus E_\partial(A).
\]
It is easy to see
\[
2|E_O(A)|+|E_\partial(A)|+|Lk(A)|\,=\,3|F(A)|.
\]
We have
\[
\begin{aligned}
&2|A|-|F(A)|+|Lk(A)|-|A\cap V_\partial|\\
=&\,2|A|-|F(A)|+|Lk(A)|-|A\cap V_\partial|+3|F(A)|-\big(\,2|E_O(A)|+|E_\partial(A)|+|Lk(A)|\,\big)\\
=&\,2|V_O(A)|-2|E_O(A)|+2|F(A)|+|V_\partial(A)|-|E_\partial(A)|,
\end{aligned}
\]
where $V_\partial(A)=A\cap V_\partial$, $V_O(A)=A\setminus V_\partial$. Note that  $V_O(A)$, $E_O(A)$ and $F(A)$ are the sets of vertices, edges and triangles of  $G(A)\setminus \partial S$. Hence
\[
|V_O(A)|-|E_O(A)|+|F(A)|\,=\,\chi(G(A)\setminus \partial S).
\]
Similarly,
\[
|V_\partial(A)|-|E_\partial(A)|\,=\,\chi(G(A)\cap \partial S).
\]
Combining the above relations, then
\[
2|A|-|F(A)|+|Lk(A)|-|A\cap V_\partial|\,=\,2\chi(G(A)\setminus \partial S)+\chi(G(A)\cap\partial S).
\]
Thus the proposition is proved.
\end{proof}

Another useful combinatorial fact is Proposition \ref{P-4-3}. Before presenting the proof, let us establish the following proposition.

\begin{proposition}\label{P-5-1}
The Euler characteristic $\chi(G(A)\cap\partial S)$ is equal to the opposite of the number of open arc components of $G(A)\cap\partial S$.
\end{proposition}
\begin{proof}
Observe that $G(A)\cap\partial S$ is an 1-manifold. Let $\lambda_1,\cdots,\lambda_m,\gamma_{m+1},\cdots,\gamma_s$ be all its components, where $\lambda_1,\cdots,\lambda_m$ are open arcs. Then
\[
\chi(G(A)\cap\partial S)\,=\,\sum\nolimits_{i=1}^m\chi(\lambda_i)\,+\sum\nolimits_{j=m+1}^s\chi(\gamma_j)=\,\sum\nolimits_{i=1}^m (-1)\,+\sum\nolimits_{j=m+1}^s 0\,=\,-m.
\]
The statement follows.
\end{proof}

\begin{proof}[\textbf{Proof of Proposition \ref{P-4-3}}]
First, let us consider the ``if" part. To show that $K[\varphi]$ is \textbf{attainable}, for any proper non-empty subset $A$ of $V$, one needs to check that
\begin{equation}\label{E-18}
\sum\nolimits_{v\in A\cap V_\partial}\varphi(v)\,>\,-\sum\nolimits_{(e,v)\in Lk(A)}\big(\pi-\Theta(e)\big)+2\pi\chi(G(A)\setminus \partial S) +\pi\chi(G(A)\cap \partial S).
\end{equation}
Without loss of generality, assume that $G(A)\setminus \partial S$ is connected and is of topological type $(g_0,n_0)$. Then $n_0\geq 1$ and $Lk(A)$ is non-empty. Meanwhile, Proposition \ref{P-5-1} implies that $\chi(G(A)\cap \partial S)$ is equal to $-n_\partial (A)$, where $n_\partial (A)$ is the number of open arc components of $G(A)\cap \partial S$.

If $g_0\geq 1$ or $n_0\geq 2$, then $\chi(G(A)\setminus \partial S)=2-2g_0-n_0\leq 0$. Noting that $\varphi\geq 0$ and $n_\partial (A)\geq 0$, the inequality (\ref{E-18}) trivially holds.

Let $g_0=0$, $n_0=1$. Thus $G(A)\setminus \partial S$ is  simply-connected and $\chi(G(A)\setminus \partial S)=1$. Suppose that $Lk(A)=\{(e_l,v_{i_l}\}_{l=1}^s$. To prove (\ref{E-18}), we need to show
\begin{equation}\label{E-19}
\sum\nolimits_{v\in A\cap V_\partial}\varphi(v)\,>\,-\sum\nolimits_{l=1}^s\big(\pi-\Theta(e_l)\big)+2\pi-\pi n_\partial (A).
\end{equation}
We divide it into the following cases:
\begin{itemize}
\item[$(i)$] $n_\partial (A)\geq 2$. The inequality (\ref{E-19}) trivially holds.
\item[$(ii)$] $n_\partial (A)=0$. Then either $G(A)\cap \partial S=\emptyset,$ or $G(A)\setminus \partial S$ is bounded by $G(A)\cap \partial S$. Because $G(A)\setminus \partial S$ is simply-connected, the latter case occurs if and only if $(g,n)=(0,1)$ and $A=V$, which contradicts to the assumption that $A$ is a proper subset of $V$. Thus $G(A)\cap \partial S=\emptyset$, and we have $A\cap V_\partial=\emptyset$. Moreover, the edges $e_1,e_2,\cdots,e_s$ form a \textbf{pseudo-Jordan} curve so that $A$ is a non-empty \textbf{enclosing set}. According  to the condition \textbf{(C2)}, one obtains
\[
\sum\nolimits_{v\in A\cap V_\partial}\varphi(v)\,=\,0\,>\,\sum\nolimits_{l=1}^s \Theta(e_l)-(s-2)\pi\,=\,-\sum\nolimits_{l=1}^s\big(\pi-\Theta(e_l)\big)+2\pi-\pi n_\partial (A).
\]
\item[$(iii)$] $n_\partial (A)= 1$. Then the edges $e_1,e_2,\cdots,e_s$ form a \textbf{semi pseudo-Jordan} arc so that $A$ is a non-empty \textbf{semi enclosing} set. Under the condition \textbf{(BV2)}, we have
\[
\sum\nolimits_{v\in A\cap V_\partial}\varphi(v)+\sum\nolimits_{l=1}^s \big(\pi-\Theta(e_l)\big)
\,>\,\pi\,=\,2\pi-\pi n_\partial (A),
\]
which implies (\ref{E-19}).
\end{itemize}

It remains to prove the ``only if" part. Assume that $e_1,e_2,\cdots,e_s$ form a \textbf{pseudo-Jordan curve} $\gamma$ with a non-empty \textbf{enclosing set} $A_\gamma$. Then $G(A_\gamma)\setminus \partial S$ is a simply-connected domain with boundary $\gamma$. Moreover, $A_\gamma\cap V_\partial=\emptyset$, $G(A_\gamma)\cap \partial S=\emptyset$.
This yields
\[
\sum\nolimits_{v\in A_\gamma\cap V_\partial}\varphi(v)\,=\,0,\quad \chi\big(G(A_\gamma)\setminus \partial S\big)\,=\,1,\quad \chi\big(G(A_\gamma)\cap \partial S\big)\,=\,0.
\]
Because $K[\varphi]$ is \textbf{attainable} and  $A_\gamma$ is a proper non-empty subset of $V$, it follows from (\ref{E-18}) that
\[
0\,\,>\,-\sum\nolimits_{(e,v)\in Lk(A_\gamma)}\big(\pi-\Theta(e)\big)+2\pi\,=\,\sum\nolimits_{l=1}^s \Theta(e_l)-(s-2)\pi.
\]
Thus \textbf{(C2)} is satisfied.

Suppose that $e_1,e_2,\cdots,e_s$ form a \textbf{semi pseudo-Jordan} arc $\gamma$ with a non-empty \textbf{semi enclosing set} $W_\gamma$. Similarly, we have
\[
\chi\big(G(W_\gamma)\setminus \partial S\big)\,=\,1,\quad \chi\big(G(W_\gamma)\cap \partial S\big)\,=\,1.
\]
Note that $W_\gamma$ is a proper non-empty subset of $V$. Using (\ref{E-18}), one obtains
\[
\sum\nolimits_{v\in W_\gamma\cap V_\partial}\varphi(v)\,>\,-\sum\nolimits_{(e,v)\in Lk(W_\gamma)}\big(\pi-\Theta(e)\big)+\pi\,=\,-\sum\nolimits_{l=1}^s \big(\pi-\Theta(e_l)\big)+\pi.
\]
That means \textbf{(BV2)} is also satisfied. We finish the proof.
\end{proof}

\section{Acknowledgements}
H. Ge is supported by NSF of China (No.11501027 and No.11871094). B. Hua is supported by NSF of China (No.11831004). Z. Zhou is supported by NSF of China
(No.11601141 and No.11631010) and China Scholarship Council (No. 201706135016).

\noindent Huabin Ge, hbge@bjtu.edu.cn\\[2pt]
\emph{Department of Mathematics, Beijing Jiaotong University, Beijing 100044, P.R. China}\\[2pt]

\noindent Bobo Hua, bobohua@fudan.edu.cn\\[2pt]
\emph{School of Mathematical Sciences, LMNS, Fudan University, Shanghai 200433, P.R. China}\\[2pt]

\noindent Ze Zhou, zhouze@hnu.edu.cn\\[2pt]
\emph{Institute of Mathematics, Hunan University, Changsha, 410082, P.R. China}
\end{document}